\documentclass[11pt]{article}

\usepackage[utf8]{inputenc}
\usepackage[a4paper,left=2.5cm,right=2.5cm,top=3cm,bottom=4cm]{geometry}
\usepackage{verbatim}
\usepackage{amsmath}
\usepackage{amssymb}
\usepackage{amsthm}
\usepackage{enumitem}
\usepackage{dsfont}
\usepackage{mathtools}
\usepackage{graphicx}
\usepackage{tikz-cd} 
\usepackage{mdframed}
\usepackage{pgfplots}
\usepackage{mathtools}
\usepackage[hidelinks]{hyperref}
\pgfplotsset{compat = newest}
\usepackage{tkz-euclide}
\DeclareMathOperator{\supp}{supp} 
\newcommand{\D}{\text{D}}
\newcommand{\Z}{\mathbb Z}
\newcommand{\dd}{\mathrm{d}}
\newcommand{\R}{\mathbb R}

\newcommand{\Rp}{\mathbb R_{> 0}}
\newcommand{\Rpz}{\mathbb R_{\geq 0}}
\newcommand{\alphahat}{\Hat{\alpha}}
\newcommand{\bhat}{\Hat{b}}
\newcommand{\sigmahat}{\Hat{\sigma}}
\newcommand{\lambdahat}{\Hat{\lambda}}
\newcommand{\What}{\Hat{W}}
\newcommand{\N}{\mathbb N}

\newcommand{\projpf}{\mathfrak{p}_*}
\definecolor{dorange}{RGB}{179, 98, 0}
\definecolor{mblue}{RGB}{0, 0, 255}

\newcommand{\erfc}{\operatorname{erfc}}
\newtheorem{maintheorem}{Theorem}

\newtheorem{lemma}{Lemma}[section]
\newtheorem{theorem}[lemma]{Theorem}
\newtheorem{proposition}[lemma]{Proposition}
\theoremstyle{definition}

\newtheorem{remark}[lemma]{Remark}
\newcommand{\RN}[1]{
  \textup{\uppercase\expandafter{\romannumeral#1}}_\epsilon
}

\usepackage{enumitem}
\newlist{steps}{enumerate}{1}
\setlist[steps, 1]{label = \textit{Step \arabic*:}}

\numberwithin{equation}{section}

\title{Positive Lyapunov Exponent in the Hopf Normal Form with Additive Noise}
\author{Dennis Chemnitz\thanks{Freie Universität Berlin - Institut für Mathematik, Arnimallee 6, 14195 Berlin}\\\href{mailto:dennis.chemnitz2@fu-berlin.de}{\small{\texttt{dennis.chemnitz2@fu-berlin.de}}}\\\small{orcid: }\href{https://orcid.org/0000-0002-3303-3533}{\small{0000-0002-3303-3533}}\and Maximilian Engel\footnotemark[1]\\
\href{mailto:maximilian.engel@fu-berlin.de}{\small{\texttt{maximilian.engel@fu-berlin.de}}}\\
\small{orcid: }\href{https://orcid.org/0000-0002-1406-8052}{\small{0000-0002-1406-8052}}}
\date{\today}

\begin{document}

\maketitle
\begin{abstract}
We prove the positivity of Lyapunov exponents for the normal form of a Hopf bifurcation, perturbed by additive white noise, under sufficiently strong shear strength. This completes a series of related results for simplified situations which we can exploit by studying suitable limits of the shear and noise parameters. The crucial technical ingredient for making this approach rigorous is a result on the continuity of Lyapunov exponents via Furstenberg-Khasminskii formulas.
\end{abstract}

\section{Introduction}

The understanding and detection of chaotic properties has been a central theme of dynamical systems theory over the past decades. Particular interest has been devoted to proving positive Lyapunov exponents in nonuniformly hyperbolic regimes \cite{Young_2013}, as an indicator of chaotic structures. Since such endeavours have turned out to be tremendously difficult for purely deterministic systems, as, for instance, the standard map \cite{Gorodetski}, more and more attention has been given to random systems where noise can help to render  chaotic features visible \cite{BlumenthalXueYoung17, BlumenthalXueYoung18, Young08}.

A particular mechanism for creating chaotic attractors under, potentially random, perturbations has been suggested and studied by Wang, Young and co-workers and has become known as shear-induced chaos \cite {linyoung, LuWangYoung, wang-young}. The main idea is to perturb limit cycles in the radial direction, where the amplitude of the perturbation depends on the angular coordinates along the limit cycle, such that a shear force in the form of radius-dependent angular velocity can lead to a stretch-and-fold mechanism in combination with overall volume contraction.
In the situation of random perturbations, this has contributed to a particular view on stochastic Hopf bifurcation, complementary to previous studies \cite{Arnoldetal96, Baxendale94}.

In more detail, the following model of a Hopf normal form with additive white noise has been studied in \cite{deville, DoanEngeletal}, also drawing attention from applications to~e.g.~laser dynamics \cite{Wieczorek}:
\begin{equation}\label{sde:intro}
    \dd \begin{pmatrix}Z_1(t)\\Z_2(t)\end{pmatrix} = \left[\begin{pmatrix}\alpha & -\beta \\ \beta &\alpha\end{pmatrix}  - \|Z(t)\|^2\begin{pmatrix}a & -b \\ b & a\end{pmatrix}\right] \begin{pmatrix}Z_1(t)\\Z_2(t)\end{pmatrix}\dd t + \sigma \dd \begin{pmatrix}W_1(t) \\ W_2(t)\end{pmatrix},
\end{equation}
where  $\sigma \geq 0$ is the strength of the noise, $\alpha \in \R$ is a parameter equal to the real
part of eigenvalues of the linearization of the vector field at (0, 0), $b \in \R$ represents shear
strength, $a > 0$, $\beta \in \R$, and $(W(t))_{t \in \Rpz}$ is a 2-dimensional Brownian motion.
We will focus on the case $\alpha > 0$, such that the system without noise ($\sigma =0$) possesses a limit cycle with radius $\sqrt{\alpha a^{-1}}$.

Deville \emph{et al} \cite{deville} showed that, in the limits of small noise and small shear,
the largest Lyapunov exponent $\lambda(\alpha, \beta, a, b, \sigma)$ for system~\eqref{sde:intro} is negative.
Doan \emph{et al} \cite{DoanEngeletal} extended these stability results to parts of the global parameter
space and proved that the random attractor for the associated random dynamical system is a singleton, establishing exponentially fast synchronization of
almost all trajectories.

Based on numerical investigations, it was conjectured in several works \cite{deville, DoanEngeletal, linyoung, Wieczorek}
that large enough shear in combination with noise may cause Lyapunov exponents to turn positive, leading to chaotic random dynamical behaviour without synchronization.
Wang and Young \cite{wang-young} obtained a proof of shear-induced chaos with deterministic instantaneous periodic driving. 
Lin and Young \cite{linyoung} introduced a simpler, affine linear  SDE model that still retains the important features of~\eqref{sde:intro} and exhibits favorable scaling properties of the parameter-dependent (numerically computed) Lyapunov exponents. Using a slight modifcation of the noise, Engel \emph{et al} \cite{engel-sdlc} obtained an analytical proof of positive Lyapunov exponents for this kind of simplified model in cylindrical coordinates. For that they used a Furstenberg-Khasminskii formula in terms of a function $\Psi$ (cf.~Figure~\ref{fig:muplot} and Theorem~\ref{theo:LE-sdlc} below), based on results in \cite{ImkellerLederer}.
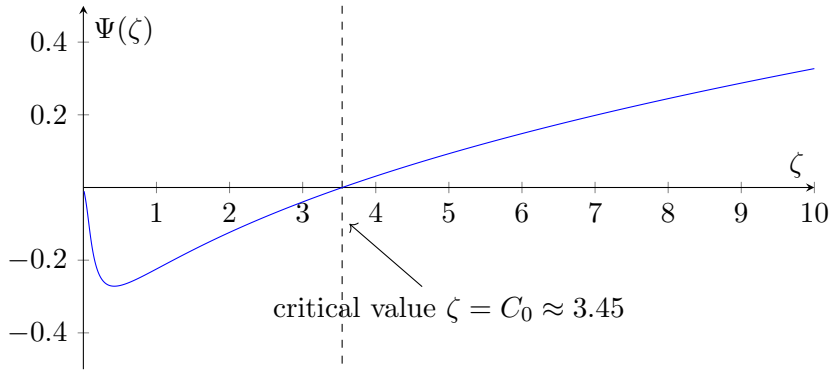
\begin{figure}[t]
    \centering
    \begin{tikzpicture}
 
\begin{axis}[
    xmin = 0, xmax = 10,
    ymin = -0.5, ymax = 0.5,
    axis x line= center,
    axis y line= center,
    xtick distance = 1,
    ytick distance = 0.2,
    width = 0.7\textwidth,
    height = 0.4\textwidth,
    xlabel = {$\zeta$},
    ylabel = {$\Psi(\zeta)$},]
\draw[-, thin, black, dashed] (3.54,0.5)--(3.54,-0.5);
 \addplot[
    thin,
    blue
] file[skip first] {muplot.dat};
\node[anchor=south](comment) at (axis cs: 5,-0.4) {critical value $\zeta = C_0 \approx 3.45$};
\draw[->] (comment)--(3.64, -0.1);
\end{axis}
 
\end{tikzpicture}
    \caption{The function $\Psi$ plotted on the interval $\zeta \in (0,10].$ The data were obtained by numerical integration.}
    \label{fig:muplot}
\end{figure}
\subsection{Main result}
The key insight of our work presented here is that the simplified cyclinder model in \cite{engel-sdlc} can be found as a large shear, small noise limit for system~\eqref{sde:intro}. This allows us to finally prove the existence of positive Lyapunov exponents by using the corresponding results in \cite{engel-sdlc} and an argument concerning the continuity of Lyapunov exponents. Hence, the main result can be expressed by associating the limit for $\lambda(\alpha, \beta, a, \epsilon^{-1} b, \epsilon \sigma)$ with the explicit Furstenberg-Khasminskii formula in terms of the function $\Psi$ (see Figure~\ref{fig:muplot}), yielding positive Lyapunov exponents for sufficiently large shear and small noise. 
\begin{maintheorem}\label{theo:main}
For all $\alpha, a, \sigma \in \Rp$ and $\beta, b \in \R$, the largest Lyapunov exponent of system~\eqref{sde:intro} satisfies
\begin{equation}\label{eq:mainLimit}
    \lim_{\epsilon \to 0} \lambda(\alpha, \beta, a, \epsilon^{-1}b, \epsilon\sigma) = 2\alpha ~\Psi\left( \frac{b^2\sigma^2}{2\alpha^2a} \right).
\end{equation}
In particular, there is a constant $C_0 \approx 3.45$ such that 
$$\lambda(\alpha, \beta, a, \epsilon^{-1}b, \epsilon\sigma)>0,$$
whenever $b^2\sigma^2>2C_0\alpha^2a$
and $\epsilon>0$ is sufficiently small, depending on $\alpha, \beta, a, b$ and $\sigma$. 
\end{maintheorem}
We remark that the situation of positive Lyapunov exponents allows for several conclusions concerning the nature of the random attractor $\{A(\omega)\}_{\omega \in \Omega}$ ($\Omega$ denotes the canonical Wiener space here), as established in \cite{DoanEngeletal}. 
In this reference, the authors identified $A(\omega) = \supp(\mu_{\omega})$, 
where $\mu_{\omega}$ denote the disintegrations of the invariant measure $\mu$ for the random dynamical system, corresponding with the stationary measure $\rho$ for the SDE~\eqref{sde:intro}.
Now, when $\lambda > 0$, we may deduce that $\mu_{\omega}$ is atomless almost surely by an extension of results by Baxendale \cite[Remark 4.12]{baxendale} to the non-compact setting \cite[Theorem 5.1.1]{EngelPhD}. 
Furthermore, by applying results on Pesin's formula for random dynamical systems in $\R^d$ \cite{biskamp}, one obtains positive metric entropy with respect to the invariant measure $\mu$ whenever $\lambda > 0$ (see also \cite[Corollary 5.2.10]{EngelPhD}). The fact that the disintegrations $\mu_{\omega}$ are SRB measures should follow by a similar extension of Ledrappier's and Young's work \cite{ledrappier-young} to the non-compact state space case.
\begin{figure}
    \centering
    \includegraphics[width = 0.8\textwidth]{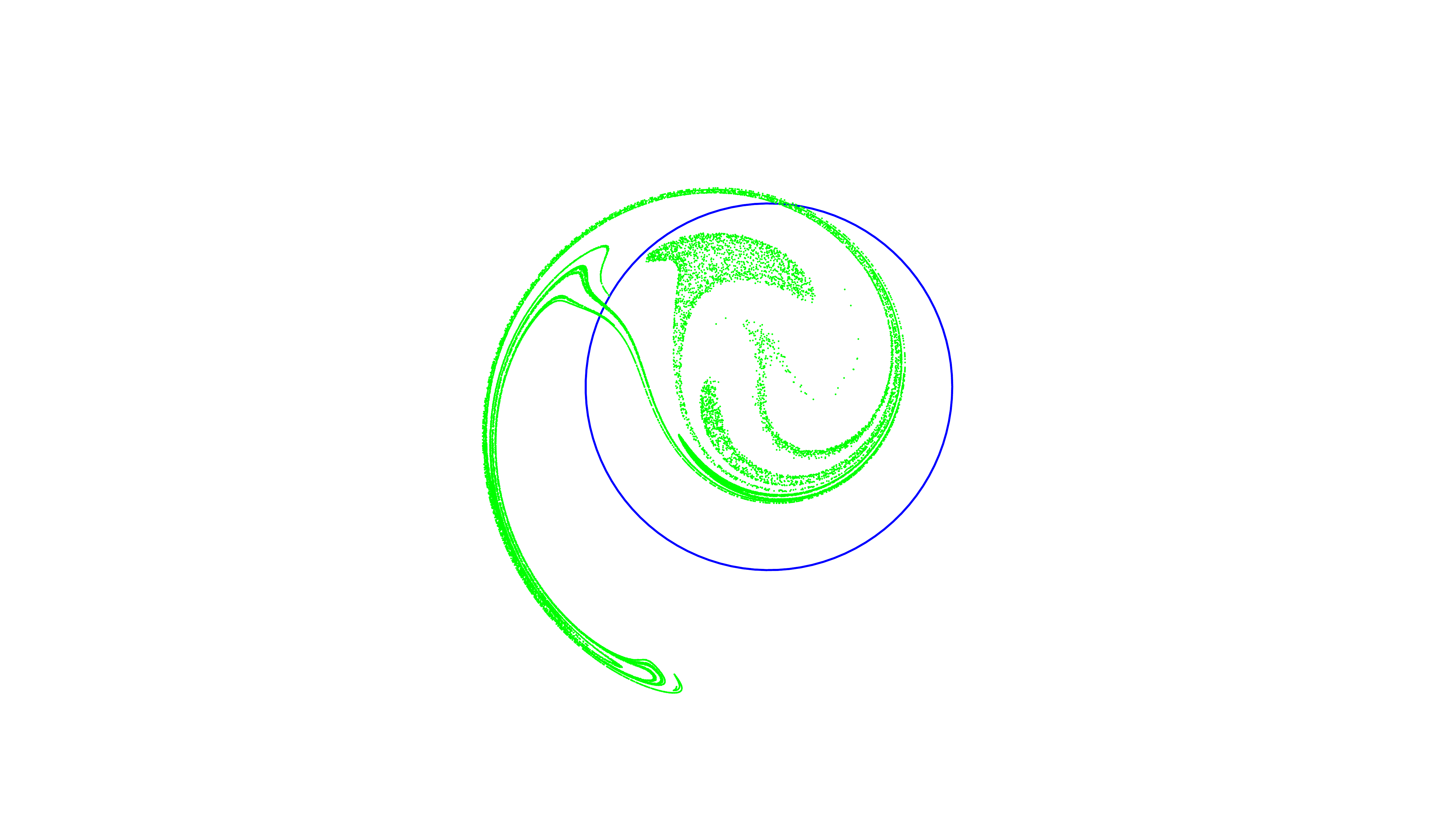}
    \caption{Chaotic random attractor for \eqref{sde:intro} with parameters $\alpha = 1$, $\beta = 1$, $a = 1$, $b = -10$ and $\sigma = 1$. The deterministic limit cycle $\left\{z \in \mathbb R^2 : \|z\| = \sqrt{\alpha a^{-1}}\right\}$ is shown in blue for reference. 
    The plot was obtained by taking 50,000 samples from the stationary distribution of \eqref{sde:intro} and evolving them numerically using an Euler-Maruyama scheme for a fixed time $T \approx 500$ with a fixed realization of the noise.
    }
    \label{fig:randattr}
\end{figure}

\subsection{Structure of the paper}
The proof of Theorem~\ref{theo:main} is presented in section~\ref{sec:fullproof},  based on a formal first derivation in Section~\ref{sec:largeshear}. In more detail, Section~\ref{sec:sdlc} recalls the bifurcation from synchronization to chaos, indicated by a change of sign of the largest Lyapunov exponent, in the simplified model in cylindrical coordinates studied in~\cite{engel-sdlc}. In Section~\ref{sec:formal derivation}, we provide a rigorous framework to study model~\eqref{sde:intro} as a random dynamical system, also introdcuing the corresponding Lyapunov exponents. Based on a specific geometric insight, we introduce the crucial, new  coordinates for the variational process along trajectories of~\eqref{sde:intro} and the change of variables $b= \epsilon^{-1} b'$ and $\sigma = \epsilon \sigma'$, yielding a description of the system that allows for obtaining the simplified model in a formal $\epsilon \to 0$ limit.

As a first step towards a rigorous proof,
Section~\ref{sec:auxiliary_processes} derives SDEs and coefficient estimates for auxiliary processes, parametrizing the projective bundle process, in order to provide Furstenberg-Khasminskii formulas that make the RDS and its linearization associated to~\eqref{sde:intro}, 
indexed by $\epsilon$, 
comparable to the simplified SDE model.
In Section~\ref{sec:stationarymeas}, we establish the unique stationary measures for the respective bundle processes along with concentration bounds and a weak convergence result as $\epsilon \to 0$. 
After showing in Section~\ref{sec:fk_formula} the Furstenberg-Khasminkii formula for the limiting process, obtained in the simplified model, we collect the previous description of the bundle processes and their invariant measures to prove, in Section~\ref{sec:continuity_LEs}, the limit of the largest Lyapunov exponent as stated in Theorem~\ref{theo:main}.
Thereby, we have to use two different coordinate systems for the variational process, one allowing for an arbitrarily close approximation of the deterministic limit cycle and one controlling the polar coordinate singularity at radius $r=0$.

\subsection{Additional relations to other work and outlook}
Lately, Breden and Engel \cite{BredenEngel} proved the existence of positive \emph{conditioned Lyapunov exponents}, as introduced in~\cite{EngelLambRasmussenTransAMS}, 
for model~\eqref{sde:intro} restricted to a bounded domain around the deterministic limit cycle. 
The setting considers limits of trajectories, conditioned on not having hit the boundary of the domain. The proof involves computer-assistance, using interval arithmetic, for obtaining an approximation of the relevant quasi-ergodic measure in a modified Furstenberg-Khasminskii formula, together with a rigorous error estimate.
Naturally, the proof procedure is always based on choosing particular values for the parameters, i.e.~in the end one can show the result only for a finite number of parameter combinations. 

In recent years, Bedrossian, Blumenthal and Punshon-Smith have embarked on a program to make random dynamical systems theory fruitful for relating chaotic stochastic dynamics with positive Lyapunov exponents to turbulent fluid flow, in particular in the form of Lagrangian chaos and passive scalar turbulence \cite{BedrossianetalBatchelor, BedrossianetalLangrange}. On a related note, the same authors have developed a new method for obtaining lower bounds for positive Lyapunov exponents, using an identity resembling Fisher information and and an adapted hypoellitpic regularity theory \cite{Bedrossianetal2022}. They have applied this method to \emph{Euler-like} systems, including a stochastically forced version of Lorenz 96, where an energy and volume conserving system is weakly perturbed, in the sense of small scaling, by linear damping and noise. Note that our situation also considers small noise perturbations but, being a dissipative system, strong damping; hence, the mechanism leading to chaos and positive Lyapunov exponents is fundamentally different, requiring a strong shear force interacting with noise and dissipation.

This leads to the general question of whether the reduction of the perturbed normal form~\eqref{sde:intro} to the simplified cylinder models in \cite{engel-sdlc, linyoung}, yielding the positivity of Lyapunov exponents, can be extended to a broad, maybe even universal, class of oscillators (van der Pol, FitzHugh-Nagumo etc.). 
We leave it as an open question for the future to find a such rigorous generalization for the phenomenon of shear-noise induced chaos.

\section{Large shear, small noise limit and relation to simplified model}
\label{sec:largeshear}
In this section, we will first give an overview over the simplified model studied by Engel et.~al.~\cite{engel-sdlc}, inspired by Lin and Young \cite{linyoung}. We will then provide a formal, non-rigorous derivation of our main result, which shall serve as a guide for the rigorous proof given in the final section.

\subsection{A bifurcation in the simplified model}\label{sec:sdlc}
As a slightly modified version of a model studied in \cite{linyoung}, Engel et.~al.~\cite{engel-sdlc} have investigated the SDE
\begin{equation}\label{sde:ogsdlc}
    \begin{cases}
    \begin{aligned}
    \dd \hat S(t) &= - \alphahat \hat S(t) \dd t + \sigmahat \left[\sin\left(\hat \Theta (t)\right) \dd \hat W_1(t) + \cos\left(\hat \Theta (t)\right) \dd \hat W_2(t)\right],\\
    \dd \hat \Theta (t) &= \left[1- \bhat \hat S(t)\right]\dd t,
    \end{aligned}
    \end{cases}
\end{equation}
where $(\hat S(t))$ is a real-valued process and $(\hat \Theta(t))$ is an $S^1 := \R /(2 \pi \Z)$- valued process. Furthermore, $\hat \alpha, \hat \sigma \in \Rp$ are positive real parameters, $\bhat \in \R$ is a real parameter and $(\hat W_1(t), \hat W_2(t))$ are independent Brownian motions. Note that the parameters $\alphahat, \bhat$ and $\sigmahat$ have roles that are similar to their respective counterparts $\alpha, b$ and $\sigma$ in equation~\eqref{sde:intro}. 
Based on previous work by Imkeller and Lederer \cite{ImkellerLederer}, it was shown in~\cite{engel-sdlc} that, depending on the values of the parameters $\alphahat, \bhat$ and $\sigmahat$, the top Lyapunov exponent for \eqref{sde:ogsdlc} can attain both positive and negative values. 

We will start by giving a brief overview of the formal setup for studying solutions of  \eqref{sde:ogsdlc} from a random dynamical systems viewpoint.
Due to the Lipschitz-continuity of all terms on the right-hand side of (\ref{sde:ogsdlc}), this equation generates a differentiable random dynamical system \cite[Definition 1.1.3]{Arnold}, which can be constructed as follows. Set $\Omega = \mathcal C^0(\Rpz, \R^2)$, where $\mathcal C^0(\Rpz, \R^2)$ denotes the space of all continuous functions $\omega: \Rpz \to \R^2$ satisfying $\omega(0) = 0$. We equip $\Omega$ with the compact-open topology and let $\mathbb P$ be the Wiener measure defined on the Borel-measurable subsets of $\Omega$. Furthermore we denote by $\varsigma: \Rpz \times \Omega \to \Omega$ the shift action given by
$$\varsigma(t,\omega)(s) = \omega(t+s) - \omega(t).$$
The Brownian motions $\hat W_{i}$, $i=1,2$, in (\ref{sde:ogsdlc}) will be interpreted as random variables defined by
$$\hat W_1 (t) := \omega_1(t) \text{ and }\hat W_2 (t) := \omega_2(t).$$
There exists a stochastic flow map $\hat \varphi: \Omega \times \Rpz \times (\R \times S^1) \to \R \times S^1$ induced by (\ref{sde:ogsdlc}) which has the following properties.
\begin{enumerate}
    \item[i)] For each $(\hat S_0, \hat \Theta_0) \in \R \times S^1$, the stochastic process $(\hat S(t), \hat \Theta(t))$ defined by 
    $$(\hat S(t), \hat \Theta(t)) := \hat \varphi(t, \cdot, \hat S_0, \hat \Theta_0)$$
    is a strong solution to \eqref{sde:ogsdlc} with initial condition $(\hat S(0), \hat \Theta(0)) = (\hat S_0, \hat \Theta_0)$.
    \item[ii)] After possibly restricting to a $\varsigma$-invariant subset $\Tilde \Omega \subseteq \Omega$ of full $\mathbb P$-measure, the skew-product $(\varsigma, \hat \varphi)$ forms a continuously differentiable random dynamical system (RDS) in the sense of \cite[Definition 1.1.3]{Arnold}. In particular the cocycle property 
    \begin{equation}\label{eq:hatcocycle}
        \hat \varphi(s+t, \omega, \hat S_0, \hat \Theta_0) = \hat \varphi(t, \varsigma(s,\omega), \hat \varphi(s, \omega, \hat S_0, \hat \Theta_0))
    \end{equation}
    
    holds for every $s,t \in \Rpz$, $\omega \in \Tilde\Omega$ and $(\hat S_0, \hat \Theta_0) \in \R \times S^1$.
\end{enumerate}
For simplicity of notation, we will identify $\Tilde \Omega$ and $\Omega$ and will use $\hat \varphi$ to mean its restriction to $\Tilde \Omega$. This allows us to define a linear map $\hat \Phi: \Rpz \times \Omega \times (\R \times S^1) \to \R^{2\times 2}$ by
$$\hat \Phi(t, \omega, \hat S_0, \hat \Theta_0) = \D_{(\hat S_0, \hat \Theta_0)} \hat \varphi(t, \omega, \hat S_0, \hat \Theta_0).$$
The cocycle property of $\hat \Phi$ over the RDS $(\varsigma, \hat \varphi)$, i.e.
$$\hat \Phi(s+t, \omega, \hat S_0, \hat \Theta_0) = \hat \Phi(t, \varsigma(s,\omega), \hat \varphi(s, \omega, \hat S_0, \hat \Theta_0)) \hat \Phi(s, \omega, \hat S_0, \hat \Theta_0),$$
comes as a direct consequence of applying the chain rule to (\ref{eq:hatcocycle}).
By the Furstenberg-Kesten Theorem, the integrability condition of which can be easily verified for our situation, the top Lyapunov exponent $\hat \lambda(\alphahat, \bhat, \sigmahat)$ can now be defined as 
$$\hat \lambda(\alphahat, \bhat, \sigmahat) := \lim_{t \to \infty} \frac{1}{t} \log\|\hat \Phi(t, \omega, \hat S_0, \hat \Theta_0)\|.$$
A useful tool for studying Lyapunov exponents is the variational process given by 
$$\begin{pmatrix}\hat s(t)\\ \hat \theta (t)\end{pmatrix} := \hat \Phi(t, \omega, \hat S_0, \hat \Theta_0) \begin{pmatrix}\hat s_0\\ \hat \theta_0\end{pmatrix},$$
for some initial condition $(\hat s_0, \hat \theta_0) \neq (0,0)$. This process satisfies the so-called variational equation (see e.g. \cite[Theorem 2.3.32]{Arnold}), which in our case takes the form
\begin{equation}
\begin{cases}
\begin{aligned}\label{sde:sdlc}
    \dd \hat s(t) &= -\alphahat \hat s(t) \dd t + \sigmahat\hat \theta(t) \dd \What_3(t),\\
    \dd \hat \theta(t) &= -\bhat \hat s(t) \dd t,
\end{aligned}
\end{cases}
\end{equation}
obtained by differentiating the coefficients of the original SDE~\eqref{sde:ogsdlc}.
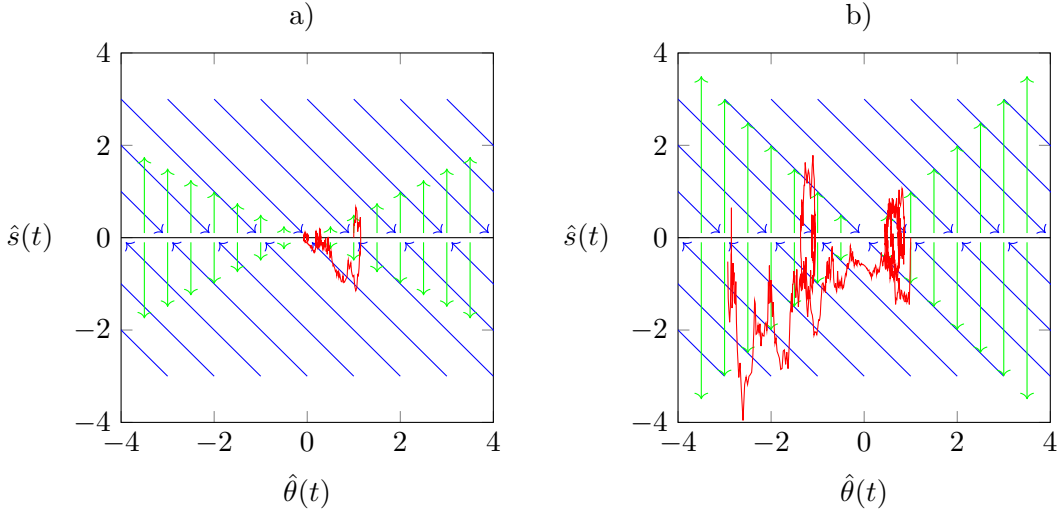
\begin{figure}[t!]\label{fig:sdlc}
    \centering
 \begin{minipage}{0.45\textwidth}
 \centering
 ~~~~~~~~~~a)\\
 \begin{tikzpicture}
 
\begin{axis}[
    xmin = -4, xmax = 4,
    ymin = -4, ymax = 4,
    xtick distance = 2,
    ytick distance = 2,
    width = 0.9\textwidth,
    height = 0.9\textwidth,
    xlabel = {$\hat \theta(t)$},
    ylabel = {$\hat s(t)$},
    ylabel style={rotate=-90}]
\draw[-,thin,black](-4,0)--(4,0);

 \draw[->, thin, blue](-6,3)--(-3.1,0.1);
 \draw[->, thin, blue](-5,3)--(-2.1,0.1);
 \draw[->, thin, blue](-4,3)--(-1.1,0.1);
 \draw[->, thin, blue](-3,3)--(-0.1,0.1);
 \draw[->, thin, blue](-2,3)--(0.9,0.1);
 \draw[->, thin, blue](-1,3)--(1.9,0.1);
 \draw[->, thin, blue](0,3)--(2.9,0.1);
 \draw[->, thin, blue](1,3)--(3.9,0.1);
 \draw[->, thin, blue](2,3)--(4.9,0.1);
 \draw[->, thin, blue](3,3)--(5.9,0.1);
 
 \draw[->, thin, blue](-3,-3)--(-5.9,-0.1);
 \draw[->, thin, blue](-2,-3)--(-4.9,-0.1);
 \draw[->, thin, blue](-1,-3)--(-3.9,-0.1);
 \draw[->, thin, blue](0,-3)--(-2.9,-0.1);
 \draw[->, thin, blue](1,-3)--(-1.9,-0.1);
 \draw[->, thin, blue](2,-3)--(-0.9,-0.1);
 \draw[->, thin, blue](3,-3)--(0.1,-0.1);
 \draw[->, thin, blue](4,-3)--(1.1,-0.1);
 \draw[->, thin, blue](5,-3)--(2.1,-0.1);
 \draw[->, thin, blue](6,-3)--(3.1,-0.1);
 
\draw[->, thin, green](0.5,0.1)--(0.5,0.25);
 \draw[->, thin, green](1,0.1)--(1,0.5);
 \draw[->, thin, green](1.5,0.1)--(1.5,0.75);
 \draw[->, thin, green](2,0.1)--(2,1);
 \draw[->, thin, green](2.5,0.1)--(2.5,1.25);
 \draw[->, thin, green](3,0.1)--(3,1.5);
 \draw[->, thin, green](3.5,0.1)--(3.5,1.75);
 \draw[->, thin, green](-0.5,0.1)--(-0.5,0.25);
 \draw[->, thin, green](-1,0.1)--(-1,0.5);
 \draw[->, thin, green](-1.5,0.1)--(-1.5,0.75);
 \draw[->, thin, green](-2,0.1)--(-2,1);
 \draw[->, thin, green](-2.5,0.1)--(-2.5,1.25);
 \draw[->, thin, green](-3,0.1)--(-3,1.5);
 \draw[->, thin, green](-3.5,0.1)--(-3.5,1.75);

\draw[->, thin, green](0.5,-0.1)--(0.5,-0.25);
 \draw[->, thin, green](1,-0.1)--(1,-0.5);
 \draw[->, thin, green](1.5,-0.1)--(1.5,-0.75);
 \draw[->, thin, green](2,-0.1)--(2,-1);
 \draw[->, thin, green](2.5,-0.1)--(2.5,-1.25);
 \draw[->, thin, green](3,-0.1)--(3,-1.5);
 \draw[->, thin, green](3.5,-0.1)--(3.5,-1.75);
 \draw[->, thin, green](-0.5,-0.1)--(-0.5,-0.25);
 \draw[->, thin, green](-1,-0.1)--(-1,-0.5);
 \draw[->, thin, green](-1.5,-0.1)--(-1.5,-0.75);
 \draw[->, thin, green](-2,-0.1)--(-2,-1);
 \draw[->, thin, green](-2.5,-0.1)--(-2.5,-1.25);
 \draw[->, thin, green](-3,-0.1)--(-3,-1.5);
 \draw[->, thin, green](-3.5,-0.1)--(-3.5,-1.75);
 
 \addplot[
    thin,
    red
] file[skip first] {sigma1.dat};
\end{axis}
 
\end{tikzpicture}
 \end{minipage}
\begin{minipage}{0.45\textwidth}

\centering
~~~~~~~~~~b)\\
\begin{tikzpicture}
 
\begin{axis}[
    xmin = -4, xmax = 4,
    ymin = -4, ymax = 4,
    xtick distance = 2,
    ytick distance = 2,
    width = 0.9\textwidth,
    height = 0.9\textwidth,
    xlabel = {$\hat \theta(t)$},
    ylabel = {$\hat s(t)$},
    ylabel style={rotate=-90}]
\draw[-,thin,black](-4,0)--(4,0);

 \draw[->, thin, blue](-6,3)--(-3.1,0.1);
 \draw[->, thin, blue](-5,3)--(-2.1,0.1);
 \draw[->, thin, blue](-4,3)--(-1.1,0.1);
 \draw[->, thin, blue](-3,3)--(-0.1,0.1);
 \draw[->, thin, blue](-2,3)--(0.9,0.1);
 \draw[->, thin, blue](-1,3)--(1.9,0.1);
 \draw[->, thin, blue](0,3)--(2.9,0.1);
 \draw[->, thin, blue](1,3)--(3.9,0.1);
 \draw[->, thin, blue](2,3)--(4.9,0.1);
 \draw[->, thin, blue](3,3)--(5.9,0.1);
 
 \draw[->, thin, blue](-3,-3)--(-5.9,-0.1);
 \draw[->, thin, blue](-2,-3)--(-4.9,-0.1);
 \draw[->, thin, blue](-1,-3)--(-3.9,-0.1);
 \draw[->, thin, blue](0,-3)--(-2.9,-0.1);
 \draw[->, thin, blue](1,-3)--(-1.9,-0.1);
 \draw[->, thin, blue](2,-3)--(-0.9,-0.1);
 \draw[->, thin, blue](3,-3)--(0.1,-0.1);
 \draw[->, thin, blue](4,-3)--(1.1,-0.1);
 \draw[->, thin, blue](5,-3)--(2.1,-0.1);
 \draw[->, thin, blue](6,-3)--(3.1,-0.1);
 
  \draw[->, thin, green](0.5,0.1)--(0.5,0.5);
 \draw[->, thin, green](1,0.1)--(1,1);
 \draw[->, thin, green](1.5,0.1)--(1.5,1.5);
 \draw[->, thin, green](2,0.1)--(2,2);
 \draw[->, thin, green](2.5,0.1)--(2.5,2.5);
 \draw[->, thin, green](3,0.1)--(3,3);
 \draw[->, thin, green](3.5,0.1)--(3.5,3.5);
 \draw[->, thin, green](-0.5,0.1)--(-0.5,0.5);
 \draw[->, thin, green](-1,0.1)--(-1,1);
 \draw[->, thin, green](-1.5,0.1)--(-1.5,1.5);
 \draw[->, thin, green](-2,0.1)--(-2,2);
 \draw[->, thin, green](-2.5,0.1)--(-2.5,2.5);
 \draw[->, thin, green](-3,0.1)--(-3,3);
 \draw[->, thin, green](-3.5,0.1)--(-3.5,3.5);

\draw[->, thin, green](0.5,-0.1)--(0.5,-0.5);
 \draw[->, thin, green](1,-0.1)--(1,-1);
 \draw[->, thin, green](1.5,-0.1)--(1.5,-1.5);
 \draw[->, thin, green](2,-0.1)--(2,-2);
 \draw[->, thin, green](2.5,-0.1)--(2.5,-2.5);
 \draw[->, thin, green](3,-0.1)--(3,-3);
 \draw[->, thin, green](3.5,-0.1)--(3.5,-3.5);
 \draw[->, thin, green](-0.5,-0.1)--(-0.5,-0.5);
 \draw[->, thin, green](-1,-0.1)--(-1,-1);
 \draw[->, thin, green](-1.5,-0.1)--(-1.5,-1.5);
 \draw[->, thin, green](-2,-0.1)--(-2,-2);
 \draw[->, thin, green](-2.5,-0.1)--(-2.5,-2.5);
 \draw[->, thin, green](-3,-0.1)--(-3,-3);
 \draw[->, thin, green](-3.5,-0.1)--(-3.5,-3.5);
 
 \addplot[
    thin,
    red
] file[skip first] {sigma2.dat};
\end{axis}
 
\end{tikzpicture}
\end{minipage}

 
    \caption{Typical trajectories (red) for $(\hat s(t), \hat \theta(t))$, $t \in [0,10]$, starting in $(\hat s_0, \hat \theta_0) = (0,1)$, with parameters $\Hat{\alpha} = \Hat{b} = \Hat{\sigma} =1$ (a) and $\Hat{\alpha} = \Hat{b} =1$, $\Hat{\sigma} = 2$ (b). The blue arrows indicate the drift field and the green arrows the noise field. Note that for a) we have $\zeta := \bhat^2\sigmahat^2\alphahat^{-3} = 1 < C_0$, while for b) we have $\zeta = 4 > C_0$.}
\end{figure}
Here, the process $(\What_3(t))$ is given by 
$$\dd\What_3(t) := \cos\left(\hat \Theta(t)\right)\dd \What_1(t) - \sin\left(\hat \Theta(t)\right)\dd \What_2(t).$$
By Levy's criterion, see e.g. \cite[Theorem IV.3.6]{RevuzYor}, the process $(\What_3 (t))$ is again a Brownian motion. 
\begin{remark}\label{rem:indLaw}
Note that $(\What_3(t))$ does not only depend on $\omega$ but also on the initial condition $(\hat S_0, \hat \Theta_0)$ for the original equation (\ref{sde:ogsdlc}). Thus, the process $(\hat s(t), \hat \theta(t))$ also depends not only on its own initial condition $(\hat s_0, \hat \theta_0)$, but also on $(\hat S_0, \hat \Theta_0)$. However, the law of $(\What_3 (t))$ is always that of a Brownian motion, independently of $(\hat S_0, \hat \Theta_0)$; hence, the law of $(\hat s(t), \hat \theta(t))$ also does not depend on $(\hat S_0, \hat \Theta_0)$ (but of course still on $(\hat s_0, \hat \theta_0)$).
\end{remark}

It follows from results in \cite{ImkellerLederer} that, for any initial condition $(\hat s_0, \hat \theta_0) \neq (0,0)$, we almost surely have
\begin{equation}
\label{eq:topLyap_conv}
\lambdahat(\alphahat, \bhat, \sigmahat) = \lim_{t \to \infty} \frac{1}{t}\log\sqrt{\hat s(t)^2 + \hat \theta(t)^2}.
\end{equation}
Thus the Lyapunov exponent $\lambdahat(\alphahat, \bhat, \sigmahat)$ is fully determined by the law of the SDE (\ref{sde:sdlc}).
\begin{remark}
For each $(\hat S_0, \hat \Theta_0) \in \R \times S^1$ and almost every $\omega \in \Omega$ there will still exist a one-dimensional subspace $V(\omega, \hat S_0, \hat \Theta_0)\subset \R^2$ such that
$$(\hat s_0, \hat \theta_0) \in V(\omega, \hat S_0, \hat \Theta_0) \Rightarrow \lim_{t \to \infty} \frac{1}{t}\log\sqrt{\hat s(t)^2 + \hat \theta(t)^2} = \lambdahat_2(\alphahat, \bhat, \sigmahat) < \lambdahat(\alphahat, \bhat, \sigmahat),$$
where $\lambdahat_2(\alphahat, \bhat, \sigmahat)$ is the second Lyapunov exponent. However, due to a result from \cite{ImkellerFlags} based on Hörmander's condition, one can show that the distribution of $V(\cdot, \hat S_0, \hat \Theta_0)$ in the projective space $\R \mathbb P^1$ is atomless. Thus, for each fixed $(\hat S_0, \hat \Theta_0, \hat s_0, \hat \theta_0)$ we have for almost every $\omega \in \Omega$
$$(\hat s_0, \hat \theta_0) \notin V(\omega, \hat S_0, \hat \Theta_0),$$
such that the equality~\eqref{eq:topLyap_conv} holds.
\end{remark}
It turns out that one can make use of homogenities of system \eqref{sde:sdlc}
to simplify the analysis of the sign of the top Lyapunov exponent. Substituting $\hat \theta(t) \mapsto \gamma \hat \theta(t)$ for some $\gamma > 0$ gives the identity
$$\lambdahat(\alphahat, \gamma \bhat, \gamma^{-1}\sigmahat) = \lambdahat(\alphahat, \bhat, \sigmahat).$$
Similarly, substituting $t \mapsto \delta t$ gives
$$\lambdahat(\delta\alphahat, \delta\bhat, \sqrt{\delta}\sigmahat) = \delta \lambdahat(\alphahat, \bhat, \sigmahat).$$
Together, these two identities show that the sign of $\lambdahat(\alphahat, \bhat, \sigmahat)$ will only depend on the value of $\zeta := \bhat^2\sigmahat^2\alphahat^{-3}$. In \cite{ImkellerLederer}, 
Imkeller and Lederer derived a semi-explicit formula for $\lambdahat(\alphahat, \bhat, \sigmahat)$. To give this formula, we first define for each $\zeta > 0$ a function $m_\zeta : \Rp \to \Rpz$ by
$$m_\zeta(u) := K_\zeta^{-1} \frac{1}{\sqrt{u}} \exp\left(-\frac{1}{\zeta}\left(\frac{1}{6}u^3-\frac{1}{2}u\right)\right).$$
Here $K_\zeta>0$ is a normalization constant given by
$$K_\zeta := \int_0^\infty \frac{1}{\sqrt{u}} \exp\left(-\frac{1}{\zeta}\left(\frac{1}{6}u^3-\frac{1}{2}u\right)\right) \dd u,$$
such that $m_\zeta$ becomes the density of a probability distribution on $\Rp$. Next, we define a function $\Psi: \Rp \to \R$ by
\begin{equation}\label{def:mu}
    \Psi(\zeta) := \frac{1}{2}\left(\int_0^\infty u~ m_\zeta(u)\dd u -1\right).
\end{equation}
Now one can give the following formula for the top Lyapunov exponent of the RDS induced by equation~\eqref{sde:ogsdlc}.
\begin{theorem}[{\cite[Theorem 3]{ImkellerLederer}, see also \cite[Theorem 2.1]{engel-sdlc}}]\label{theo:LE-sdlc}
For any initial condition $(\hat s(0), \hat \theta(0))\neq (0,0)$ the solution to (\ref{sde:sdlc}) satisfies
$$\lim_{t \to \infty} \frac{1}{t}\log\left(\sqrt{\hat s(t)^2+\hat \theta(t)^2}\right) = \lambdahat(\alphahat, \bhat, \sigmahat) = \alphahat~ \Psi\left(\frac{\bhat^2\sigmahat^2}{\alphahat^3}\right).$$
\end{theorem}

In particular, this means that the sign of $\lambdahat(\alphahat, \bhat, \sigmahat)$ is determined by the sign of $\Psi(\zeta)$, where $\zeta := \bhat^2\sigmahat^2\alphahat^{-3}$. Engel et.~al.~used this fact to show the following bifurcation result.
\begin{theorem}[{\cite[Theorem 1.1]{engel-sdlc}}]\label{theo:bif}
The function $\Psi : \Rp \to \R$ has a unique zero at $C_0\approx 3.45$. Furthermore, we have $\Psi(\zeta) < 0$ for $\zeta < C_0$ and $\Psi(\zeta)>0$ for $\zeta>C_0$. In particular, the top Lyapunov exponent satisfies
$$\lambdahat(\alphahat, \bhat, \sigmahat) \begin{cases}>0, \quad \emph{ if } \bhat^2\sigmahat^2\alphahat^{-3} > C_0 \\
=0, \quad \emph{ if } \bhat^2\sigmahat^2\alphahat^{-3} = C_0 \\
<0, \quad \emph{ if } \bhat^2\sigmahat^2\alphahat^{-3} < C_0.
\end{cases}$$
\end{theorem}
\begin{remark}
Our constant $C_0 \approx 3.45$ is different from the constant $c_0 \approx 0.2823$ from \cite{engel-sdlc}, but related by $C_0 := c_0^{-1}$. This is due to a slight reformulation of the statement.
\end{remark}
\subsection{A formal derivation of the large shear, small noise limit}\label{sec:formal derivation}
Now we turn our attention again to the Hopf normal form with additive noise, which we will write as 
\begin{equation}\label{sde:main}
    \dd Z(t) = F(Z(t))\dd t + \sigma \dd W(t),
\end{equation}
where $F: \R^2 \to \R^2$ is given by
$$F(Z):= \begin{pmatrix}\alpha & -\beta \\ \beta &\alpha\end{pmatrix}Z  - \|Z\|^2\begin{pmatrix}a & -b \\ b & a\end{pmatrix}Z.$$
Recall that $\alpha, a$ and $\sigma$ are positive real parameters, $\beta$ and $b$ are real parameters and $(W(t))_{t \in \Rpz}$ is a 2-dimensional Brownian motion. We will consider the parameters $\alpha, \beta, a, b$ and $\sigma$ as fixed for the time being.
The formal setup for studying this equation is very similar to the one described in the previous subsection. System \eqref{sde:main} has already been investigated from a random dynamical systems viewpoint in \cite{deville} and \cite{DoanEngeletal} and we will build on their results. As before, the probability space is given by $\Omega = \mathcal C^0(\Rpz, \R^2)$, with Wiener measure $\mathbb P$ and $\mathbb P$-invariant time shift denoted by $\varsigma: \Rpz \times \Omega \to \Omega$, again identifying the process $(W(t))$ with random variables
$W(t) := \omega(t).$ It has been shown in \cite[Theorem A]{DoanEngeletal}, that \eqref{sde:main} induces a random dynamical system in the same sense as in the previous subsection, i.e.~that there exists a map $\varphi: \Rpz \times \Omega \times \R^2 \to \R^2$ satisfying the following properties:
\begin{enumerate}
    \item[i)] For each $Z_0 \in \R^2$, the stochastic process $(Z(t))$ defined by 
    $$Z(t) := \varphi(t, \cdot, Z_0)$$
    is a strong solution to (\ref{sde:main}) with initial condition $Z(0)= Z_0$.
    \item[ii)] After possibly restricting to a $\varsigma$-invariant subset $\Tilde \Omega \subseteq \Omega$ of full $\mathbb P$-measure, the skew-product $(\varsigma, \varphi)$ forms a continuously differentiable random dynamical system in the sense of \cite[Definition 1.1.3]{Arnold}. In particular, the cocycle property
    $$\varphi(s+t, \omega, Z_0) = \varphi(t, \varsigma(s, \omega), \varphi(s, \omega, Z_0))$$
    holds for every $s,t \in \Rpz$, $\omega \in \Tilde \Omega$ and $Z_0 \in \R^2$.
\end{enumerate}
As before, we define the linearization $\Phi: \Rpz \times \Omega \times \R^2 \to \R^{2\times 2}$ by 
$$\Phi(t, \omega, Z_0) := \D_{Z_0} \varphi(t, \omega, Z_0),$$
which satisfies the identity 
$$\Phi(s+t, \omega, Z_0) = \Phi(t, \varsigma(s, \omega), \varphi(s, \omega, Z_0)) \Phi(s, \omega, Z_0).$$
Since the integrability condition of the Furstenberg-Kesten Theorem is verified \cite[Proposition 4.1]{DoanEngeletal}, we can define the top Lyapunov $\lambda(\alpha, \beta, a, b, \sigma)$ by
$$\lambda(\alpha, \beta, a, b, \sigma) := \lim_{t \to \infty} \frac{1}{t} \log \|\Phi(t, \omega, Z_0)\|.$$
Analogously to the previous subsection, we introduce the variational process $(Y(t))$ defined by 
$$Y(t):= \Phi(t, \cdot, Z(t))Y_0$$
for some initial condition $Y_0 \neq (0,0)$. The corresponding variational equation is given by the linear random ordinary differential equation
\begin{align}\label{sde:vari}
    \dd Y(t) =& [\D F(Z(t))] Y(t) \dd t\nonumber \\
    =& \left[\begin{pmatrix}\alpha & -\beta \\ \beta & \alpha \end{pmatrix} - \|Z(t)\|^2\begin{pmatrix}a & -b \\ b & a \end{pmatrix} - 2\begin{pmatrix}a & -b \\ b & a \end{pmatrix}Z(t)Z(t)^T\right]Y(t) \dd t.
\end{align}
It was shown in \cite{deville} that the top Lyapunov exponent can be expressed by the almost-sure identity 
\begin{equation}
\label{eq:Lyap_Y}
 \lambda(\alpha, \beta, a, b, \sigma) = \lim_{t \to \infty} \frac{1}{t} \log \|Y(t)\|,   
\end{equation}
which holds independently of the chosen initial conditions $Z_0 \in \R^2$ and $Y_0 \in \R^2\setminus \{0\}$.

One of the defining features of equation \eqref{sde:main} is the rotational symmetry of the drift term. Thus, it is reasonable to express the system in polar coordinates. Given the solution $Z(t)$ to~\eqref{sde:main}, we consider the $\Rpz$-valued process $(r(t))$ and the $\R/(2\pi\Z)$-valued process $(\phi(t))$ uniquely defined by 
\begin{equation}\label{eq:polardef}
    \begin{pmatrix}Z_1(t)\\ Z_2(t)\end{pmatrix} = r(t) \begin{pmatrix}\cos(\phi(t))\\\sin(\phi(t))\end{pmatrix},~ \forall t\geq 0.
\end{equation}
\begin{proposition}\label{prop:polarSDE}
The processes $(r(t))$ and $(\phi(t))$ satisfy the It\^o-SDEs
\begin{equation}
    \begin{cases}
    \begin{aligned}\label{sde:polar1}
    \dd r(t) &= \left(\alpha r(t) - a r(t)^3 + \frac{\sigma^2}{2r(t)}\right) \dd t + \sigma\left[\cos(\phi(t))\dd W_1(t) + \sin(\phi(t)) \dd W_2(t)\right],\\
    \dd \phi(t) &= \left(\beta -br(t)^2\right)\dd t +\frac{\sigma}{r(t)}\left[-\sin(\phi(t))\dd W_1(t) + \cos(\phi(t))\dd W_2(t)\right].
    \end{aligned}
    \end{cases}
\end{equation}
\end{proposition}
\begin{proof}
See Appendix \ref{app:polar}. The same derivation can also be found in \cite{deville}.
\end{proof}
We now introduce two real-valued processes $(W_r(t))$ and $(W_\phi(t))$ uniquely defined by $W_r(0) = W_\phi(0) = 0$ and the It\^o-SDEs
\begin{equation*}
    \begin{cases}
    \begin{aligned}
    \dd W_r(t) =& \cos(\phi(t))\dd W_1(t) + \sin(\phi(t)) \dd W_2(t),\\
    \dd W_\phi(t) =&-\sin(\phi(t))\dd W_1(t) + \cos(\phi(t))\dd W_2(t).
\end{aligned}
    \end{cases}
\end{equation*}
Again, using Levy's criterion (see e.g. \cite[Theorem IV.3.6]{RevuzYor}), it can be easily checked that $(W_r(t))$ and $(W_\phi(t))$ are independent Brownian motions. We can now rewrite (\ref{sde:polar1}) as
\begin{equation}
    \begin{cases}
    \begin{aligned} \label{sde:polar2}
    \dd r(t) &= \left(\alpha r(t) - a r(t)^3 + \frac{\sigma^2}{2r(t)}\right) \dd t + \sigma \dd W_r(t)\\
    \dd \phi(t) &= \left(\beta -br(t)^2\right)\dd t +\frac{\sigma}{r(t)}\dd W_\phi(t).
\end{aligned}
    \end{cases}
\end{equation}
We will also conduct a change of coordinates for the variational process $(Y(t))$. In order to simplify the equation, we express  $(Y(t))$ in an orthonormal basis that is adjusted to the polar representation of $Z(t)$ (see Fig. \ref{fig:polarplot}). 
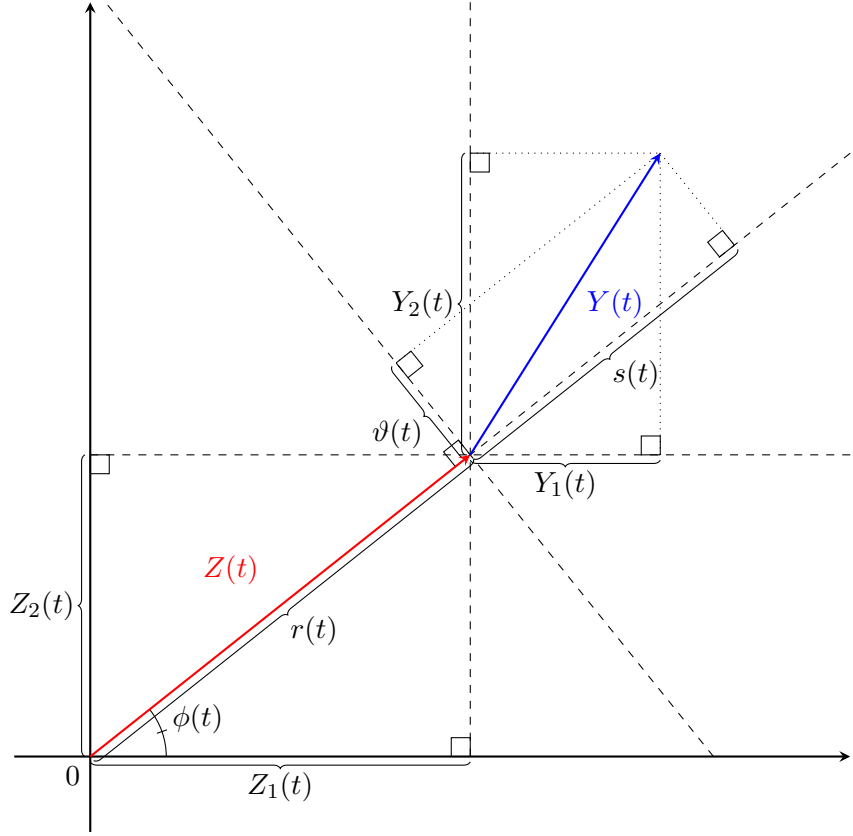
\begin{figure}[t!]
    \centering
    \begin{tikzpicture} 
        \draw[-stealth, thick] (-1,0)--(10,0);
        \draw[-stealth, thick] (0,-1)--(0,10);
        \coordinate[label = below left:$0$] (O) at (0,0);
        \coordinate (Z) at (5,4);
        \coordinate (Zx) at (5,0);
        \coordinate (Zy) at (0,4);

        \draw[dashed] (Zx)--(Z);
        \draw[dashed] (Zy)--(Z);
        \tkzMarkAngle (Zx,O,Z)
        \tkzLabelAngle[pos=1.5](Zx,O,Z){$\phi(t)$}
        
        \coordinate (Y) at (7.5,8);
        \coordinate (Yx) at (7.5,4);
        \coordinate (Yy) at (5,8);
       
        \draw[dotted] (Yx)--(Y);
        \draw[dotted] (Yy)--(Y);
        
        \coordinate (Z2) at (10,8);
        \coordinate (sl1) at (8.2,0);
        \coordinate (sl2) at (0.2,10);
        \coordinate (Zx2) at (5,10);
        \coordinate (Zy2) at (10,4);
        \coordinate (s) at ($(Z)!(Y)!(Z2)$);
        \coordinate (theta) at ($(sl1)!(Y)!(sl2)$);
        \draw[dashed] (sl1)--(sl2);
        \draw[dashed] (Zx2)--(Z);
        \draw[dashed] (Zy2)--(Z);
        \draw[dashed] (Z2)--(Z);
        \draw[dotted] (Y)--(s);
        \draw[dotted] (Y)--(theta);
        
        \tkzMarkRightAngle (Z,Zx,O)
        \tkzMarkRightAngle (O,Zy,Z)
        \tkzMarkRightAngle (sl2,Z,O) 
        \tkzMarkRightAngle (Y,s,Z)
        \tkzMarkRightAngle (Z,Yy,Y)
        \tkzMarkRightAngle (Y,Yx,Z)
        \tkzMarkRightAngle (Z,theta,Y)
        
        \draw[decorate, decoration = {brace,raise=2pt}] (Zx)--(O)
        node [midway, below = 2pt] {$Z_1(t)$};
        \draw[decorate, decoration = {brace,raise=2pt}] (O)--(Zy)
        node [midway, left=2pt] {$Z_2(t)$};
        \draw[decorate, decoration = {brace,raise=2pt}] (Z)--(O)
        node [midway, below right] {$r(t)$};
        \draw[decorate, decoration = {brace,raise=2pt}] (Yx)--(Z)
        node [midway, below = 2pt] {$Y_1(t)$};
        \draw[decorate, decoration = {brace,raise=2pt}] (Z)--(Yy)
        node [midway, left=2pt] {$Y_2(t)$};
        \draw[decorate, decoration = {brace,raise=2pt}] (s)--(Z)
        node [midway, below right] {$s(t)$};
        \draw[decorate, decoration = {brace,raise=2pt}] (Z)--(theta)
        node [midway, below left] {$\vartheta(t)$};
        
        \draw [-stealth, red, thick] (O)--(Z) node[midway,label = above left:{$Z(t)$}] (r){};
        \draw[-stealth, blue, thick] (Z)--(Y)node[midway,label = right:{$Y(t)$}] (r){};
    \end{tikzpicture}
    \caption{Geometric construction of the different coordinates used. The processes $(r(t))$ and $(\phi(t))$ are the usual polar coordinates for $(Z(t))$. Formally, the variational process $(Y(t))$ has values $Y(t) \in T_{Z(t)}\R^2 \simeq \R^2$, where $T\R^2$ denotes the (trivial) tangent bundle of $\R^2$. Now the pair $(s(t), \vartheta(t))$ expresses the same point in $T_{Z(t)}\R^2$ in a time-dependent orthonormal basis. In particular the process $(s(t))$ describes radial perturbations of $Z(t)$, while $(\vartheta(t))$ describes angular perturbations of Z(t).}
    \label{fig:polarplot}
\end{figure}
For that purpose, we introduce two real-valued processes $(s(t))$ and $(\vartheta(t))$ given by
$$s(t) := \begin{pmatrix}\cos(\phi(t)) \\ \sin(\phi(t))\end{pmatrix}^T Y(t)$$
and
$$\vartheta(t) := \begin{pmatrix}-\sin(\phi(t)) \\ \cos(\phi(t))\end{pmatrix}^T Y(t).$$

Clearly the norm is unchanged, i.e.~we have
$$\|Y(t)\| = \sqrt{Y_1(t)^2+Y_2(t)^2} = \sqrt{s(t)^2  + \vartheta(t)^2},$$
for all $t\geq 0$. Due to the rotational symmetry of our original system (\ref{sde:main}), this basis change also simplifies the stochastic differential equations for the variational process.
\begin{proposition}\label{prop:variSDE}
The processes $(s(t))$ and $(\vartheta(t))$ satisfy the Stratonovich-SDEs
\begin{equation}
    \begin{cases}
    \begin{aligned}\label{sde:varipolar}
    \dd s(t) &= \left(\alpha-3ar(t)^2\right)s(t)\dd t+ \frac{\sigma}{r(t)} \vartheta(t) \circ \dd W_{\phi}(t)\\
    \dd\vartheta(t) &= \left(\alpha- ar(t)^2\right) \vartheta(t) \dd t - 2br(t)^2 s(t)\dd t- \frac{\sigma}{r(t)} s(t) \circ \dd W_\phi(t).
\end{aligned}
    \end{cases}
\end{equation}
\end{proposition}
\begin{proof}
See Appendix \ref{app:polar}.
\end{proof}
Note that the SDE is given in Stratonovich form, which will be used throughout the rest of the paper whenever it fits the given situation best, as is the case here.
Recall that we are interested in studying the large shear, small noise limit. 
To that end, we consider from now on fixed real parameters $\alpha, \beta, a, b'$ and $\sigma'$, with $\alpha, a$ and $\sigma'$ positive, and set $b := \epsilon^{-1} b'$ and $\sigma = \epsilon \sigma'$ such that we can study the limit $\epsilon \to 0$. The main object of interest for our analysis is the (joint) law of the processes $(s(t))$ and $(\vartheta(t))$ since this will be crucial to compute the Lyapunov exponent via~\eqref{eq:Lyap_Y}. Note that the process $(\phi(t))$ only enters the right-hand side of \eqref{sde:varipolar} indirectly, through $(W_\phi(t))$. 
However, by a similar argument to the one made in Remark \ref{rem:indLaw}, the law of $(s(t), \vartheta(t))$ does only depend on their initial values $(s(0), \vartheta(0))$ and the law of $(r(t))$. Since the equation for $(r(t))$ in (\ref{sde:polar2}) also does not depend on $\phi(t)$, the law of $(s(t), \vartheta(t))$ can be studied via the following ``self-contained'' system of SDEs. 
\begin{equation}
    \begin{cases}
    \begin{aligned}\label{sde:rsvartheta}
    \dd r(t) &= \left(\alpha r(t) - a r(t)^3 + \frac{\epsilon^2\sigma'^2}{2r(t)}\right) \dd t + \epsilon\sigma' \circ \dd W_r(t) \\
    \dd s(t) &= \left(\alpha-3ar(t)^2\right)s(t)\dd t+ \frac{\epsilon\sigma'}{r(t)} \vartheta(t) \circ \dd W_{\phi}(t)\\
    \dd\vartheta(t) &= \left(\alpha- ar(t)^2\right) \vartheta(t) \dd t - 2\epsilon^{-1}b'r(t)^2 s(t)\dd t- \frac{\epsilon\sigma'}{r(t)} s(t) \circ \dd W_\phi(t).
\end{aligned}
    \end{cases}
\end{equation}

At the moment taking a limit $\epsilon \to 0$ is, even formally, not possible due to the presence of a term $\epsilon^{-1}$ in the last equation. To circumvent this problem, we rescale the process $(\vartheta(t))$ by introducing the real-valued process $(\theta(t))$ as $$\theta(t) := \epsilon \vartheta(t) \text{, for all } t \geq 0.$$
This allows us to rewrite (\ref{sde:rsvartheta}) as
\begin{equation}\label{sde:rstheta}
    \begin{cases}
    \begin{aligned}
    \dd r(t) &= \left(\alpha r(t) - a r(t)^3 + \frac{\epsilon^2\sigma'^2}{2r(t)}\right) \dd t + \epsilon\sigma' \circ \dd W_r(t) \\
    \dd s(t) &= \left(\alpha-3ar(t)^2\right)s(t)\dd t+ \frac{\sigma'}{r(t)} \theta(t) \circ \dd W_{\phi}(t)\\
    \dd\theta(t) &= \left(\alpha- ar(t)^2\right) \theta(t) \dd t - 2b'r(t)^2 s(t)\dd t- \frac{\epsilon^2\sigma'}{r(t)} s(t) \circ \dd W_\phi(t).
\end{aligned}
    \end{cases}
\end{equation}
Note that, due to the estimate
\begin{equation}\label{ineq:thetavartheta}
    \epsilon\sqrt{s(t)^2+\vartheta(t)^2} \leq \sqrt{s(t)^2+\theta(t)^2} \leq \sqrt{s(t)^2+\vartheta(t)^2},
\end{equation}
which holds for sufficiently small $\epsilon$, we still have
\begin{align} \label{eq:LyapLim}
    \lambda(\alpha, \beta, a, b, \sigma) :=& \lim_{t \to \infty} \frac{1}{t}\log\left(\|Y(t)\|\right)\nonumber\\
    =& \lim_{t \to \infty} \frac{1}{t}\log\left(\sqrt{s(t)^2+\vartheta(t)^2}\right)\nonumber\\
    =& \lim_{t \to \infty} \frac{1}{t}\log\left(\sqrt{s(t)^2+\theta(t)^2}\right).
\end{align}
In equation~\eqref{sde:rstheta}, we can now set $\epsilon = 0$, at least formally. This leaves us with
\begin{equation*}
    \begin{cases}
    \begin{aligned}
    \dd r(t) &= \left(\alpha r(t) - a r(t)^3 \right) \dd t\\
    \dd s(t) &= \left(\alpha-3ar(t)^2\right)s(t)\dd t+ \frac{\sigma'}{r(t)} \theta(t)\circ \dd W_{\phi}(t)\\
    \dd\theta(t) &= \left(\alpha- ar(t)^2\right) \theta(t) \dd t - 2b'r(t)^2 s(t)\dd t.
\end{aligned}
    \end{cases}
\end{equation*}
In this singular limit, the equation for $r(t)$ is a purely deterministic ODE, which has a globally exponentially stable equilibrium at $\Hat{r}:= \sqrt{\alpha a^{-1}}$. Since we are only interested in the asymptotic behavior as $t$ tends to infinity, we may furthermore formally set $r(t) = \Hat{r}$, yielding
\begin{equation}\label{sde:sdlcAsLim}
    \begin{cases}
    \begin{aligned}
    \dd s(t) &= -2\alpha s(t)\dd t+ \sigma'\sqrt{\frac{a}{\alpha}} \theta(t) \circ \dd W_{\phi}(t),\\
    \dd\theta(t) &= - 2\frac{\alpha b'}{a} s(t)\dd t.
\end{aligned}
    \end{cases}
\end{equation}
Note that after setting $$\alphahat := 2 \alpha,~\sigmahat := \sigma' \sqrt{\frac{a}{\alpha}} \text{ and } \bhat := 2\frac{\alpha b'}{a},$$
equation~\eqref{sde:sdlcAsLim} has the same form as the variational equation \eqref{sde:sdlc} studied in section \ref{sec:sdlc} (note that It\^o and Stratonovich integrals coincide here). This means that the law of the processes $(s(t), \theta(t))$ evolving under (\ref{sde:sdlcAsLim}) and the processes $(\hat s(t), \hat \theta(t))$ evolving under (\ref{sde:sdlc}) are identical. Therefore, by Theorem \ref{theo:LE-sdlc}, the solutions to (\ref{sde:sdlcAsLim}) almost surely satisfy
\begin{equation}\label{eq:sdlcLyapLim}
    \lim_{t \to \infty}\frac{1}{t}\log\left(\sqrt{{s}(t)^2+\theta(t)^2}\right) = \alphahat~ \Psi\left(\frac{\bhat^2\sigmahat^2}{\alphahat^3}\right) = 2\alpha~ \Psi\left(\frac{b'^2\sigma'^2}{2\alpha^2a}\right)
\end{equation}
for all initial conditions $(s(0),\theta(0))\neq (0,0)$. It is thus reasonable to believe that 
$$\lim_{\epsilon \to 0} \lambda(\alpha, \beta, a, \epsilon^{-1}b', \epsilon\sigma') = 2\alpha ~\Psi\left( \frac{b'^2\sigma'^2}{2\alpha^2a} \right)$$
should hold which we will, in fact, prove in the following section.
\begin{remark}
The crucial step in this formal derivation is the coordinate change from $(s,\vartheta)$ to $(s,\theta)$. Since this is a  non-orthonormal change of coordinates of the tangent-bundle $T\mathbb R^2$, it may also be interpreted as a change in the metric tensor, and thereby the geometry of the manifold $\mathbb R^2$. In particular the change $\vartheta \mapsto \theta := \epsilon \vartheta$ has the effect of shortening (assuming $\epsilon<1$) distances in angular direction, i.e. along circles centered at the origin, while preserving distances in radial direction, i.e. along straight lines through the origin. This precisely corresponds to the geometry of a cone. Now, taking the limit $\epsilon \to 0$ has the effect of reducing the angle at the apex of the cone to $0$ or equivalently sending the apex to infinity. Geometrically this results in a cylinder, which is precisely the state space for the simplified model \eqref{sde:ogsdlc}. The singularity at the origin, which arises through the coordinate change for $\epsilon < 1$, foreshadows an obstacle (cf. section \ref{sec:continuity_LEs})  that is dealt with in the rigorous proof presented in the remainder of this paper.
\end{remark}
\section{Proof of Theorem \ref{theo:main}}
\label{sec:fullproof}
The goal of this final section is to rigorously prove the large shear, small noise limit behaviour, derived formally in the previous section.  Thereby we will establish the main result, Theorem \ref{theo:main}.
A crucial challenge is that, in general, Lyapunov exponents do not depend continuously on the underlying system (see \cite{viana} for a survey). In the case of random systems, however, continuity of Lyapunov exponents can be shown under fairly general conditions, going back to ideas by Young \cite{young_conti}. The main technique is to express the top Lyapunov exponent as an integral against the stationary distribution of an auxiliary process on the projective bundle, by the so-called Furstenberg-Khasminskii formula.

In order to make the dependence on $\epsilon$ more explicit, we denote the processes $(r(t), s(t), \theta(t))$ 
by $(r_\epsilon(t), s_\epsilon(t), \theta_\epsilon(t))$. With this notation, we rewrite equation \eqref{sde:rstheta} as
\begin{equation}\label{sde:eps}
    \begin{cases}
    \begin{aligned}
    \dd r_\epsilon(t) &= \left(\alpha r_\epsilon(t) - a r_\epsilon(t)^3 + \frac{\epsilon^2\sigma'^2}{2r_\epsilon(t)}\right) \dd t + \epsilon\sigma' \circ \dd W_r(t) \\
    \dd \begin{pmatrix}s_\epsilon(t)\\ \theta_\epsilon(t)\end{pmatrix} &= A^{(1)}(r_\epsilon(t), \epsilon)  \begin{pmatrix}s_\epsilon(t)\\ \theta_\epsilon(t)\end{pmatrix}\dd t + A^{(2)}(r_\epsilon(t), \epsilon)  \begin{pmatrix}s_\epsilon(t)\\ \theta_\epsilon(t)\end{pmatrix}\circ \dd W_\phi(t)
\end{aligned}
    \end{cases}
\end{equation}
where the matrix-valued functions $A^{(i)}: \Rp \times \Rpz \to \R^{2\times 2}$, $i=1,2$,  are given by
\begin{align}\label{eq:defiA}
    A^{(1)}(r, \epsilon) :=&\begin{pmatrix}\alpha-3ar^2& 0\\ -2b'r^2&\alpha-ar^2 \end{pmatrix} ,& 
    A^{(2)}(r,\epsilon) :=& \begin{pmatrix}0& \sigma'r^{-1}\\ -\epsilon^2\sigma'r^{-1}&0\end{pmatrix}.
\end{align}
We also rewrite the variational equation \eqref{sde:sdlc} of model~\eqref{sde:ogsdlc} as
\begin{equation}\label{sde:newsdlc}
     \dd \begin{pmatrix}\Hat{s}(t)\\ \Hat{\theta}(t)\end{pmatrix} = A^{(1)}(\Hat{r}, 0)  \begin{pmatrix}\Hat{s}(t)\\ \Hat{\theta}(t)\end{pmatrix}\dd t + A^{(2)}(\Hat{r}, 0)  \begin{pmatrix}\Hat{s}(t)\\ \Hat{\theta}(t)\end{pmatrix}\circ \dd \hat W_3(t).
\end{equation}

\subsection{Auxiliary processes for Furstenberg-Khasminskii formula}
\label{sec:auxiliary_processes}
The Furstenberg-Khasminskii formula is a way to express the top Lyapunov exponent via an integral against the stationary distribution of an induced process on the projective bundle.
We will use the processes $\psi$ and $\Lambda$, introduced generally in the following lemma, to work with such a formula in convenient coordinates.
\begin{lemma}\label{lemm:proj}
Let $(w(t))$ be a semimartingale with values in some open set $D\subset \R$, $B^{(i)}: D \to \R^{2\times 2}$, $i=1,2$, some smooth matrix-valued functions and $(v(t))$ be an $(\R^2\setminus\{0\})$-valued stochastic process satisfying the SDE
$$\dd v(t) = B^{(1)}(w(t)) v(t) \dd t + B^{(2)}(w(t)) v(t) \circ \dd W(t).$$
Consider the processes
\begin{align*}
    \psi(t) &:= 2 \tan^{-1}\left(\frac{v_2(t)}{v_1(t)}\right), & \Lambda(t) &:= \log\left(\sqrt{v_1(t)^2 + v_2(t)^2}\right).
\end{align*}
Then $(\psi(t))$ and $(\Lambda(t))$ satisfy the SDEs
\begin{align*}
      \dd \psi(t) =& p_1(w(t),\psi(t)) \dd t +p_2( w(t),\psi(t)) \circ \dd W(t),\\
      \dd \Lambda(t) =& q_1(w(t),\psi(t)) \dd t +q_2(w(t), \psi(t)) \circ \dd W(t),
\end{align*}
where the functions $p_{i}, q_{i}: D \times S^1 \to \R$, $i=1,2$, are given by
\begin{align*}
    p_i(w, \psi) =& B^{(i)}_{2,1}(w)(1+\cos(\psi))-B^{(i)}_{1,2}(w)(1-\cos(\psi)) + \left(B^{(i)}_{2,2}(w)-B^{(i)}_{1,1}(w)\right)\sin(\psi)\\
    q_i(w, \psi) =& B^{(i)}_{1,1}(w)(1+\cos(\psi))+B^{(i)}_{2,2}(w)(1-\cos(\psi)) + \left(B^{(i)}_{1,2}(w)+B^{(i)}_{2,1}(w)\right)\sin(\psi).
\end{align*}
\end{lemma}
\begin{proof}
See appendix \ref{app:polar}.
\end{proof}
\begin{remark}
A similar statement is given in \cite[Equations (5) and (6)]{ImkellerLederer}, with a slightly different parametrization of the projective space $\R \mathbb P^1 \simeq S^1$, leading to different expressions.
\end{remark}
In the spirit of Lemma \ref{lemm:proj} we define for each $\epsilon\geq 0$ an $(\R/2\pi\Z)$-valued process $(\psi_\epsilon(t))$ and an $\Rp$-valued process $(\Lambda_\epsilon(t))$ by
\begin{align*}
    \psi_\epsilon(t) &:= 2 \tan^{-1}\left(\frac{\theta_\epsilon(t)}{s_\epsilon(t)}\right), & \Lambda_\epsilon(t) &:= \log\left(\sqrt{\theta_\epsilon(t)^2 + s_\epsilon(t)^2}\right),
\end{align*}
as well as an $(\R/2\pi\Z)$-valued process $(\hat \psi (t))$ and an $\Rp$-valued process $(\hat \Lambda(t))$ by
\begin{align*}
    \hat \psi(t) &:= 2 \tan^{-1}\left(\frac{\hat \theta(t)}{\hat s(t)}\right), & \hat \Lambda(t) &:= \log\left(\sqrt{\hat\theta(t)^2 + \hat s(t)^2}\right).
\end{align*}
Here, the values of $\psi_\epsilon$ and $\hat \psi$ should be seen as a parametrization of the projective space $\R \mathbb P^1 \simeq S^1$ by $[-\pi, \pi)$.
Indeed, our definition implies
\begin{align*}
    \begin{pmatrix}s_\epsilon(t) \\ \theta_\epsilon(t)\end{pmatrix} &\in \operatorname{Span}\left\{\begin{pmatrix}\cos(\tfrac{1}{2}\psi_\epsilon(t)) \\ \sin(\tfrac{1}{2}\psi_\epsilon(t))\end{pmatrix}\right\},&  \begin{pmatrix}\hat s(t) \\ \hat \theta(t)\end{pmatrix} &\in \operatorname{Span}\left\{\begin{pmatrix}\cos(\tfrac{1}{2}\hat \psi(t)) \\ \sin(\tfrac{1}{2}\hat \psi(t))\end{pmatrix}\right\}.
\end{align*}
The values of $(t^{-1}\Lambda_\epsilon(t))$ and $(t^{-1}\hat\Lambda(t))$ are commonly called finite time Lyapunov exponents in the literature. Note that, by \eqref{eq:LyapLim}, we have 
\begin{equation}\label{eq:FTLElim}
    \lim_{t \to \infty}\frac{1}{t} \Lambda_\epsilon(t) =  \lambda(\alpha, \beta, a, \epsilon^{-1} b', \epsilon \sigma'),
\end{equation}
for each $\epsilon>0$ and, by \eqref{eq:sdlcLyapLim}, we have
$$\lim_{t \to \infty}\frac{1}{t} \hat \Lambda(t)= \lambdahat(\alphahat, \bhat, \sigmahat) = 2\alpha~ \Psi\left(\frac{b'^2\sigma'^2}{2\alpha^2a}\right).$$
Recalling~\eqref{sde:eps} and~\eqref{eq:defiA}, we introduce the functions $g_{i}, h_{i} : \Rp \times S^1 \times \Rpz \to \R$, $i=1,2$, by
\begin{equation}\label{eq:defigihi}
\begin{array}{rl}
    g_i(r,\psi ,\epsilon) :=& A^{(i)}_{2,1}(r,\epsilon)(1+\cos(\psi))-A^{(i)}_{1,2}(r,\epsilon)(1-\cos(\psi)) + (A^{(i)}_{2,2}(r,\epsilon)-A^{(i)}_{1,1}(r,\epsilon))\sin(\psi),\\
    h_i(r,\psi,\epsilon) :=& A^{(i)}_{1,1}(r,\epsilon)(1+\cos(\psi))+A^{(i)}_{2,2}(r,\epsilon)(1-\cos(\psi)) + (A^{(i)}_{1,2}(r,\epsilon)+A^{(i)}_{2,1}(r,\epsilon))\sin(\psi),
    \end{array}
\end{equation}
and the functions $g_3, h_3: \Rp \times S^1 \times \Rpz \to \R$ by 
\begin{align*}
     g_3 (r,\psi, \epsilon):=& g_1( r,\psi,\epsilon) +\frac{1}{2} g_2(r,\psi, \epsilon)\partial_\psi g_2(r,\psi, \epsilon),\nonumber \\
     h_3 (r,\psi, \epsilon):=& h_1(r,\psi, \epsilon) +\frac{1}{2} g_2(r,\psi, \epsilon)\partial_\psi h_2(r,\psi, \epsilon).
\end{align*}
\begin{proposition}\label{prop:fk} The processes $(\psi_\epsilon(t))$ and $(\Lambda_\epsilon(t))$ satisfy the SDEs
     \begin{align}
         \begin{cases}
         \begin{aligned}\label{sde:rpsi}
         \dd r_\epsilon(t) &= \left(\alpha r_\epsilon(t) - a r_\epsilon(t)^3 + \frac{\epsilon^2\sigma'^2}{2r_\epsilon(t)}\right) \dd t + \epsilon\sigma' \dd W_r(t) \\
         \dd \psi_\epsilon(t) &= g_3( r_\epsilon(t), \psi_\epsilon(t),\epsilon) \dd t + g_2(r_\epsilon(t), \psi_\epsilon(t), \epsilon) \dd W_\phi(t),\\
         \end{aligned}
         \end{cases}\\
         \dd \Lambda_\epsilon(t) = h_3(r_\epsilon(t),\psi_\epsilon(t),  \epsilon) \dd t + h_2(r_\epsilon(t), \psi_\epsilon(t), \epsilon) \dd W_\phi(t).~~\,\nonumber
     \end{align}
     The processes $(\hat \psi(t))$ and $(\hat \Lambda(t))$ satisfy the SDEs
     \begin{align}
         \dd \hat \psi(t) &=  g_3(\hat r,\hat \psi(t), 0) \dd t + g_2(\hat r, \hat \psi(t), 0) \dd \hat W_3(t),\label{sde:psihat}\\
         \dd \hat \Lambda(t) &=  h_3(\hat r, \hat \psi(t), 0) \dd t + h_2(\hat r, \hat \psi(t), 0) \dd \hat W_3(t).\nonumber
     \end{align}
\end{proposition}
\begin{proof}
The equations follow directly by applying Lemma \ref{lemm:proj} and the It\^o-Statonovish correction formula. Here, we make use of the fact that the quadratic covariation between the semimartingales $(r_\epsilon(t))$ and $(W_\phi(t))$ vanishes since $(r_\epsilon(t))$ is only driven by the Brownian motion $(W_r(t))$, which is independent of $(W_\phi(t))$.
\end{proof}
For technical reasons, we will also have to consider the logarithm of the norm of the variational process in the original coordinates $(s(t),\vartheta(t))$. For that purpose, we rewrite the SDE for $(s(t))$ and $(\vartheta(t))$ given in (\ref{sde:rsvartheta}) as
\begin{equation}\label{sde:tildeeps}
    \dd \begin{pmatrix}s_\epsilon(t)\\ \vartheta_\epsilon(t)\end{pmatrix} = \Tilde A^{(1)}(r_\epsilon(t), \epsilon)  \begin{pmatrix}s_\epsilon(t)\\ \vartheta_\epsilon(t)\end{pmatrix}\dd t + \Tilde A^{(2)}(r_\epsilon(t), \epsilon)  \begin{pmatrix}s_\epsilon(t)\\ \vartheta_\epsilon(t)\end{pmatrix}\circ \dd W_\phi(t),
\end{equation}
where the matrix-valued functions $\Tilde A^{(i)}: \Rp \times \Rp \to \R^{2\times 2}$, $i=1,2$,  are given by
\begin{align}\label{eq:defiAtilde}
    \Tilde A^{(1)}(r, \epsilon) :=&\begin{pmatrix}\alpha-3ar^2& 0\\ -2\epsilon^{-1}b'r^2&\alpha-ar^2 \end{pmatrix} ,& 
    \Tilde A^{(2)}(r,\epsilon) :=& \begin{pmatrix}0& \epsilon\sigma'r^{-1}\\ -\epsilon\sigma'r^{-1}&0\end{pmatrix}.
\end{align}
\begin{remark}\label{rem:epszero}
Note that while the SDE \eqref{sde:eps} is well-defined for $\epsilon\geq 0$, the SDE \eqref{sde:tildeeps} is only well-defined for $\epsilon>0$. This is precisly the motivation for mainly considering the coordinates $(s_\epsilon, \theta_\epsilon)$, rather than $(s_\epsilon, \vartheta_\epsilon)$.
\end{remark}
Similarly to our previous definitions, we set
\begin{align*}
    \Tilde\psi_\epsilon(t) &:= 2 \tan^{-1}\left(\frac{\vartheta_\epsilon(t)}{s_\epsilon(t)}\right), & \Tilde\Lambda_\epsilon(t) &:= \log\left(\sqrt{\vartheta_\epsilon(t)^2 + s_\epsilon(t)^2}\right).
\end{align*}
Note that (\ref{ineq:thetavartheta}) implies
\begin{equation}\label{ineq:tilde}
    |\Tilde \Lambda_\epsilon(t) - \Lambda_\epsilon(t)| \leq -\log(\epsilon),
\end{equation}
for sufficiently small $\epsilon>0$.
In addition, we define the functions $\Tilde g_{i}, \Tilde h_{i} : S^1 \times \Rp \times \Rp \to \R$, $i=1,2$, by
\begin{align}\label{eq:defigihitilde}
    \Tilde g_i(\psi, r,\epsilon) :=& \Tilde A^{(i)}_{2,1}(r,\epsilon)(1+\cos(\psi))-\Tilde A^{(i)}_{1,2}(r,\epsilon)(1-\cos(\psi)) + (\Tilde A^{(i)}_{2,2}(r,\epsilon)- \Tilde A^{(i)}_{1,1}(r,\epsilon))\sin(\psi)\nonumber,\\
    \Tilde h_i(\psi, r,\epsilon) :=& \Tilde A^{(i)}_{1,1}(r,\epsilon)(1+\cos(\psi))+\Tilde A^{(i)}_{2,2}(r,\epsilon)(1-\cos(\psi)) + (\Tilde A^{(i)}_{1,2}(r,\epsilon)+\Tilde A^{(i)}_{2,1}(r,\epsilon))\sin(\psi),
\end{align}
and a function $\Tilde g_3: S^1 \times \Rp \times \Rp \to \R$ by 
\begin{align*}
     \Tilde g_3 (r,\psi, \epsilon):= \Tilde g_1(r,\psi, \epsilon) +\frac{1}{2} \Tilde g_2(r,\psi, \epsilon)\partial_\psi \Tilde g_2(r,\psi,\epsilon).
\end{align*}
\begin{remark}
Note that we have $\Tilde h_2(r,\psi,\epsilon) = 0$, for all $r$,$\psi$ and $\epsilon$. This is precisely the case if and only if $\Tilde A^{(2)}(r, \epsilon) \in \mathfrak{so}(2)$, where $\mathfrak{so}(2)$ denotes the Lie-algebra of the Lie-group $SO(2)$ and is explicitly given by 
$$\mathfrak{so}(2) = \left\{\begin{pmatrix}0&t\\-t&0\end{pmatrix}: t \in \mathbb R\right\}.$$
Since $\Tilde{h}_2 = 0$, it is not necessary to define a function $\Tilde h_3$.
\end{remark}
Now, completely analogous to Proposition \ref{prop:fk}, we can obtain the following SDEs for $(\Tilde \psi_\epsilon)$ and $(\Tilde \Lambda_\epsilon(t))$.
\begin{proposition}\label{prop:fktilde}
 The processes $(\Tilde\psi_\epsilon(t))$ and $(\Tilde\Lambda_\epsilon(t))$ satisfy the SDEs
     \begin{align*}
         \begin{cases}
         \begin{aligned}
         \dd r_\epsilon(t) &= \left(\alpha r_\epsilon(t) - a r_\epsilon(t)^3 + \frac{\epsilon^2\sigma'^2}{2r_\epsilon(t)}\right) \dd t + \epsilon\sigma' \dd W_r(t) \\
         \dd \Tilde\psi_\epsilon(t) &= \Tilde g_3(r_\epsilon(t), \Tilde \psi_\epsilon(t), \epsilon) \dd t + \Tilde g_2(r_\epsilon(t), \Tilde \psi_\epsilon(t), \epsilon) \dd W_\phi(t),
         \end{aligned}
         \end{cases}\\
         \dd \Tilde \Lambda_\epsilon(t) = \Tilde h_1(r_\epsilon(t), \psi_\epsilon(t), \epsilon) \dd t.~~~~~~~~~~~~~~~~~~~~~~~~~~~~~~~~~~~
     \end{align*}
\end{proposition}
The precise shape of the functions $\Tilde h_1$ and $h_{i}, g_{i},, \Tilde g_{i}$, $i=1,2,3$, will not be important for the remainder of this section. Instead, we will only deploy their continuity and certain bounds they satisfy. 
We will use the notation ''$\bullet \lesssim \circ$``, meaning ''There exists a constant C, possibly depending on $\alpha, \beta, a, b'$ and $\sigma'$, but not on $r, \psi$ or $\epsilon$, such that $\bullet \leq C \circ$ holds``. 
The relevant bounds can be formulated in the following way.
\begin{lemma}\label{lemm:bounds}
The following estimates hold.
\begin{enumerate}
    \item[i)] $|h_2(r,\psi, \epsilon)| \lesssim 1+\epsilon^{-2}r^{-1}.$
    \item[ii)] $|h_3(r,\psi,  \epsilon)|\lesssim 1+r^2+\epsilon^{-4}r^{-2}$.
    \item[iii)] $|\Tilde h_1(r,\psi, \epsilon)|\lesssim 1+\epsilon^{-1}r^2.$
\end{enumerate}
\end{lemma}
\begin{proof}
From the definitions (\ref{eq:defiA}) and (\ref{eq:defiAtilde}) we get the estimates
\begin{align*}
    \|A^{(1)}(r,\epsilon)\| &\lesssim 1+r^2,&
    \|A^{(2)}(r,\epsilon)\| &\lesssim 1+\epsilon^{-2}r^{-1},&
    \|\Tilde A^{(1)}(r,\epsilon)\| &\lesssim 1+\epsilon^{-1}r^2.
\end{align*}
Here, it is most convenient to think of $\|\cdot\|$ as the maximums-norm on $\mathbb R^{2\times 2}$; however, all statements hold independently of the norm chosen. Plugging this into (\ref{eq:defigihi}) and (\ref{eq:defigihitilde}) respectively yields
\begin{align*}
    |h_1(r,\psi, \epsilon)| &\lesssim 1+r^2, & |g_2(r,\psi, \epsilon)| &\lesssim 1+\epsilon^{-2}r^{-1},&
    |h_2(r,\psi, \epsilon)| &\lesssim 1+\epsilon^{-2}r^{-1},\\
    |\partial_\psi h_2(r,\psi, \epsilon)| &\lesssim 1+\epsilon^{-2}r^{-1}, &
    |\Tilde h_1(r,\psi, \epsilon)| &\lesssim 1+\epsilon^{-1}r^2.
\end{align*}
Finally, we obtain
\begin{align*}
    |h_3(r,\psi, \epsilon)|&\leq |h_1(r,\psi, \epsilon)|+|g_2(r,\psi, \epsilon)||\partial_\psi h_2(r,\psi, \epsilon)|\\
    &\lesssim 1 + r^2 + \left(1+\epsilon^{-2}r^{-1}\right)^2\\
    &\lesssim 1+r^2+\epsilon^{-4}r^{-2}.
\end{align*}
This finishes the proof.
\end{proof}
\subsection{Stationary measures}
\label{sec:stationarymeas}
We briefly introduce some notation for the following: Let $\mathbb X$ be a topological space. We will denote by $\mathcal B_0(\mathbb X)$ the Banach space of bounded measurable functions on $\mathbb X$ and by $\mathcal C_0(\mathbb X)$ the Banach space of bounded continuous functions on $\mathbb X$, both equipped with the usual supremum norm. Furthermore, we will consider the set of probability measures $\mathcal M_1(\mathbb X)$ on $\mathbb X$, and  the push-forward $m_* : \mathcal M_1(\mathbb X) \to \mathcal M_1(\mathbb Y)$ associated to some map $m: \mathbb X \to \mathbb Y$. We say that a sequence of measures $(\nu_n)_{n \in \N}$ in $\mathcal M_1(\mathbb X)$ converges weakly to some measure $\nu \in \mathcal M_1(\mathbb X)$, if for each $\eta \in \mathcal C_0(\mathbb X)$ we have
$$\lim_{n \to \infty} \int_{\mathbb X} \eta(x) \nu_n(\dd x) = \int_{\mathbb X} \eta(x) \nu(\dd x).$$
The SDE for $(r_\epsilon(t),\psi_\epsilon(t))$ (cf.~equation~\eqref{sde:rpsi}) induces a Markov process on $\Rp\times S^1$, for each $\epsilon\geq 0$. We will denote the time-1 transition operator of this Markov-process by $\mathcal P_\epsilon: \mathcal B_0(\Rp\times S^1) \to\mathcal B_0(\Rp\times S^1)$, i.e.
$$\mathcal P_\epsilon\eta (r,\psi) := \mathbb E_{r_\epsilon(0) = r,~\psi_\epsilon(0) = \psi}\left[\eta(r_\epsilon(1),\psi_\epsilon(1))\right].$$
We denote its dual operator by $\mathcal P^*_\epsilon: \mathcal M_1(\Rp\times S^1 ) \to\mathcal M_1(\Rp\times S^1)$.
Analogously, we introduce the time-1 transition operator $\hat {\mathcal P}: \mathcal B_0(S^1) \to\mathcal B_0(S^1)$ of the Markov process induced by the SDE for $\hat \psi$ (cf.~equation~\eqref{sde:psihat}), and its dual $\hat {\mathcal P}^*: \mathcal M_1(S^1) \to\mathcal M_1(S^1)$. 
An operator $\mathcal P: \mathcal B_0(\mathbb X) \to \mathcal B_0(\mathbb X)$ has the Feller property, if it can be restricted to an operator $\mathcal P: \mathcal C_0(\mathbb X) \to \mathcal C_0(\mathbb X)$.
\begin{lemma}\label{lemm:Pphi}The following hold.
\begin{enumerate}
    \item[i)] Let $\eta \in \mathcal C_0 ( \Rp\times S^1)$. Then the map $\mathcal P_\bullet \eta: (r, \psi, \epsilon) \mapsto \left(\mathcal P_\epsilon \eta\right)(r, \psi)$ is in $C_0(\Rp \times S^1 \times \Rpz)$.
    \item[ii)] For each $\epsilon \geq 0$, the operator $\mathcal P_\epsilon: \mathcal B_0(\Rp \times S^1) \to \mathcal B_0(\Rp \times S^1)$ has the Feller property.
    \item[iii)] The operator $\hat {\mathcal P}: \mathcal B_0(S^1) \to \mathcal B_0(S^1)$ has the Feller property.
\end{enumerate}
\end{lemma}
\begin{proof}
i): The boundedness of $\mathcal P_\bullet \eta$ follows directly from the boundedness of $\eta$. Thus, we only have to show continuity. Without loss of generality, we fix some $\epsilon^* > 0$ and henceforth only consider $\epsilon \in [0, \epsilon^*]$. Next we amend the SDE \eqref{sde:rpsi} by $\dd \epsilon(t) = 0$, i.e. we consider the SDE
\begin{align}
    \begin{aligned}
    \begin{cases} \label{sde:amended}
    \dd r'(t) &= \left(\alpha r'(t) - a r'(t)^3 + \frac{\epsilon'(t)^2\sigma'^2}{2r'(t)}\right) \dd t + \epsilon'(t)\sigma' \dd W'_1(t) \\
         \dd \psi'(t) &= g_3(r'(t),\psi'(t),  \epsilon'(t)) \dd t + g_2(r'(t),\psi'(t),  \epsilon'(t)) \dd W'_2(t)\\
         \dd \epsilon'(t) &= 0
    \end{cases}
    \end{aligned}
\end{align}
on the state space $(r',\psi',  \epsilon') \in \Rp \times S^1 \times [0, \epsilon^*]$, where $(W'_1(t), W'_2(t))$ is a two-dimensional Brownian motion.
Since $(r'(t))$ has the same law as $(r_{\epsilon_0}(t))$ with $\epsilon_0 = \epsilon'(0)$, there is almost surely no blow-up in finite time and global (in time) solutions of \eqref{sde:amended} exist.
Let
\begin{equation}\label{sp:prime}
    (r'_{r_0,\psi_0,  \epsilon_0}(t), \psi'_{r_0,\psi_0,  \epsilon_0}(t),  \epsilon'_{r_0,\psi_0,  \epsilon_0}(t))_{(r_0,\psi_0,  \epsilon_0) \in \Rp \times S^1 \times [0, \epsilon^*],~t \in [0,1]},
\end{equation}
be a solution to the SDE \eqref{sde:amended} with
$$(r'_{r_0, \psi_0, \epsilon_0}(0),\psi'_{r_0, \psi_0, \epsilon_0}(0),  \epsilon'_{r_0, \psi_0, \epsilon_0}(0)) = (r_0, \psi_0, \epsilon_0).$$
We want to show that the process \eqref{sp:prime} has a modification which is continuous in $r_0$, $\psi_0$ and $\epsilon_0$. For each $n \in \mathbb N$ we define the waiting time $\tau_n(r_0,\psi_0,  \epsilon_0)$ by 
$$\tau_n(r_0,\psi_0,  \epsilon_0) := \inf \left\{t \in [0,1]:r'_{r_0, \psi_0, \epsilon_0}(t) \notin \left[n^{-1}, n\right]\right\}$$
and consider the stopped process 
\begin{equation}\label{sp:stopped}
    (r_{r_0,\psi_0,  \epsilon_0}^{(n)}(t), \psi_{r_0,\psi_0,  \epsilon_0}^{(n)}(t),  \epsilon_{r_0,\psi_0,  \epsilon_0}^{(n)}(t)) := (r'_{r_0,\psi_0,  \epsilon_0}, \psi'_{r_0,\psi_0,  \epsilon_0},  \epsilon'_{r_0,\psi_0,  \epsilon_0})(t \wedge \tau_n(r_0,\psi_0,  \epsilon_0)).
\end{equation}
Note that we have, for any $n, m \in \mathbb N$,
\begin{equation}\label{eq:stopped}
    (r_{r_0,\psi_0,  \epsilon_0}^{(n)}(t), \psi_{r_0,\psi_0,  \epsilon_0}^{(n)}(t),  \epsilon_{r_0,\psi_0,  \epsilon_0}^{(n)}(t)) = (r_{r_0,\psi_0,  \epsilon_0}^{(m)}(t), \psi_{r_0,\psi_0,  \epsilon_0}^{(m)}(t),  \epsilon_{r_0,\psi_0,  \epsilon_0}^{(m)}(t))
\end{equation}
for all $t \leq 1 \wedge \tau_n(r_0,\psi_0,  \epsilon_0) \wedge \tau_m(r_0,\psi_0,  \epsilon_0)$.
The drift and diffusion coefficients of \eqref{sde:amended} are smooth and thus bounded and Lipschitz continuous on the restricted state space $\Gamma_n := [n^{-1},n] \times S^1 \times [0, \epsilon^*]$. Therefore we can use \cite[Theorem IX.2.4]{RevuzYor} to find a modification of \eqref{sp:stopped} which is continuous in $r_0$, $\psi_0$ and $\epsilon_0$. Since any two continuous modifications are necessarily indistinguishable, the equality \eqref{eq:stopped} does still hold almost surely. Now we can define a modification of \eqref{sp:prime} by
$$(r'_{r_0,\psi_0,  \epsilon_0}(t), \psi'_{r_0,\psi_0,  \epsilon_0}(t),  \epsilon'_{r_0,\psi_0,  \epsilon_0}(t)) := \lim_{n \to \infty} (r_{r_0,\psi_0,  \epsilon_0}^{(n)}(t), \psi_{r_0,\psi_0,  \epsilon_0}^{(n)}(t),  \epsilon_{r_0,\psi_0,  \epsilon_0}^{(n)}(t)).$$
Since the sequence on the right-hand side restricted to $(r_0,\psi_0,  \epsilon_0) \in \Gamma_m$ is almost surely eventually constant for each $m \in \mathbb N$, this does indeed define a modification of \eqref{sp:prime} which is continuous in $r_0$, $\psi_0$ and $\epsilon_0$.

Note that $(r'_{r_0, \psi_0, \epsilon_0}(t),\psi'_{r_0, \psi_0, \epsilon_0}(t))$ has the same law as the solution $(r_{\epsilon_0}(t),\psi_{\epsilon_0}(t))$ of \eqref{sde:rpsi} with initial conditions $r_{\epsilon_0} (0) = r_0$ and $\psi_{\epsilon_0} (0) = \psi_0$. Thus, we get
$$\left(\mathcal P_\epsilon \eta \right)(r,\psi) = \mathbb E_{r_\epsilon(0) = r,~\psi_\epsilon(0) = \psi}\left[\eta(r_\epsilon(1),\psi_\epsilon(1))\right] = \mathbb E\left[ \eta\left(r'_{r, \psi, \epsilon}(1), r'_{r, \psi, \epsilon}(1)\right)\right]$$
and continuity of $\mathcal P_\bullet \eta$ follows from the dominated convergence theorem.

ii): Let again $\eta \in \mathcal C_0 ( \Rp\times S^1)$. Since by i) the map $\mathcal P_\bullet \eta: (r, \psi, \epsilon) \mapsto \left(\mathcal P_\epsilon \eta\right)(r, \psi)$ is in $C_0(\Rp \times S^1 \times \Rpz)$, we have, in particular, that $\mathcal P_\epsilon \eta: (r,\psi) \mapsto \left(\mathcal P_\epsilon \eta\right)(r, \psi)$ is in $C_0(\Rp \times S^1)$ for each $\epsilon$, implying the Feller property.

iii): Note that the functions $g_3(\hat r,\cdot, 0), ~g_2(\hat r,\cdot, 0) : S^1 \to \mathbb R$ appearing in equation \eqref{sde:psihat} are smooth functions defined on a compact domain and thus, in particular, bounded and Lipschitz-continuous. Thus the Feller property of $\hat{\mathcal P}$ follows directly from \cite[Theorem IX.2.5]{RevuzYor}.
\end{proof}

In the following let $\mathfrak{p}^{\Rp} : \Rp \times S^1 \to \Rp$ and $\mathfrak{p}^{S^1} : \Rp \times S^1 \to S^1$ denote the projections onto the first and second component respectively. Firstly, we establish the existence and uniqueness of stationary measures for the respective processes.
\begin{proposition}\label{prop:uniquestat}
The following hold.
\begin{enumerate}
    \item[i)] For each $\epsilon>0$, there exists a unique measure $\rho_\epsilon \in \mathcal M_1(\Rp \times S^1)$ with $\mathcal P_\epsilon^* \rho_\epsilon = \rho_\epsilon$. Its marginal on the $r$-coordinate $\projpf^{\Rp} \rho_\epsilon \in \mathcal M_1 (\Rp)$ has a density $\xi_\epsilon : \Rp \to \Rpz$ with respect to Lebesgue measure given by 
    \begin{equation}\label{eq:xi}
    \xi_\epsilon(r) := L_\epsilon r \exp\left(-\frac{a}{2\epsilon^2\sigma'^2}\left(r^2-\Hat{
    r}^2\right)^2\right),
    \end{equation}
    where $L_\epsilon\geq 0$ is a normalization constant given by
    \begin{equation*}
    \label{eq:L_epsilon}
       L_\epsilon := \int_0^\infty r \exp\left(-\frac{a}{2\epsilon^2\sigma'^2}\left(r^2-\Hat{
    r}^2\right)^2\right) \dd r = \frac{2\sqrt{2a}}{\sqrt{\pi}\epsilon\sigma'\erfc\left(-\frac{\alpha}{\epsilon \sigma'\sqrt{2a}}\right)},  
    \end{equation*}
    with $\erfc$ denoting the complementary error function.
    \item[ii)] The exists a unique measure $\hat \rho \in \mathcal M_1 (S^1)$ with $\hat {\mathcal P}^* \hat \rho = \hat \rho.$
    \item[iii)] The measure $\rho_0 = \delta_{\hat r} \times \hat \rho \in \mathcal M_1(\Rp \times S^1)$ is the unique measure satisfying $\mathcal P_0^* \rho_0 = \rho_0$.
\end{enumerate}
\end{proposition}
\begin{proof}
    For i), we will start by establishing existence of a stationary measure. One can easily check that, for each $\epsilon>0$, the function $\xi_\epsilon$ satisfies the stationary Fokker-Planck equation
    $$0 = -\partial_r\left[\left(\alpha r- ar^3 +\frac{\epsilon^2\sigma'^2}{2r}\right)\xi_\epsilon(r)\right]+\frac{\epsilon^2\sigma'^2}{2}\partial_r^2\xi_\epsilon(r)$$
    and $\xi_\epsilon(r)\dd r \in \mathcal M_1(\Rp)$ is thus a stationary distribution for the process $(r_\epsilon(t))$. Next we can consider the family of probability measures $\mathcal A_\epsilon \subset \mathcal M_1(\Rp\times S^1)$, defined by
    $$\mathcal A_\epsilon = \left\{\rho \in \mathcal M_1(\Rp\times S^1): \projpf^{\Rp}\rho(\dd r) = \xi_\epsilon(r) \dd r\right\}.$$
    Since $\xi_\epsilon(r)\dd r$ is stationary, $\mathcal A_\epsilon$ is $\mathcal P_\epsilon^*$ invariant, i.e.~we have $\mathcal P_\epsilon^*\mathcal A_\epsilon \subseteq \mathcal A_\epsilon$. Furthermore, with respect to the topology of weak convergence, $\mathcal A_\epsilon$ is convex and closed by definition and tight since $\xi_\epsilon(r)\dd r$ is tight and $S^1$ is compact. Starting from an arbitrary $\rho_\epsilon^{(0)} \in \mathcal A_\epsilon$, we can construct a sequence $\left(\rho_\epsilon^{(n)}\right)_{n \in \mathbb N_0}$ by
    $$\rho_\epsilon^{(n)} = \frac{1}{n+1} \sum_{k = 0}^n \left(\mathcal P_\epsilon^*\right)^k \rho_\epsilon^{(0)}.$$
    By invariance and convexity of $\mathcal A_\epsilon$, we have $\rho_\epsilon^{(n)} \in \mathcal A_\epsilon$ for all $n \in \mathbb N_0$. By Prokhorov's theorem, the sequence
    has an accumulation point $\rho_\epsilon$ with respect to the topology of weak convergence, which must lie in $\mathcal A_\epsilon$ due to closedness. Analagously to the proof of the Krylov-Bogolyubov theorem (see e.g. \cite[Theorem 1.10]{hairerNotes}), we can conclude that $\mathcal P_\epsilon^*\rho_\epsilon = \rho_\epsilon$. Uniqueness follows directly from the uniform ellipticity of the associated generator (see e.g. \cite[section 3]{ChekrounTantetII} for details). 
    
For ii), both existence and uniqueness were already established in \cite{ImkellerLederer}, by using the Krylov-Bogolyubov theorem for existence and Hörmander's theorem for uniqueness.

For iii), it is easy to check that $\rho_0 = \delta_{\hat r} \times \Hat \rho$ satisfies $\mathcal P_0^* \rho_0 = \rho_0$. Thus it only remains to show that $\rho_0$ is the unique stationary measure. Suppose we are given a measure $\rho_0' \in \mathcal M_1(\Rp\times S^1)$ with $\mathcal P_0^* \rho_0' = \rho_0'$. Since the process $(r_0(t))$ is governed by the ODE
$$\dd r_0(t) = \left(\alpha r_0(t) - ar_0(t)^3\right)\dd t,$$
we must have $\projpf^{\Rp} \rho_0' = \delta_{\hat r}$. Furthermore, if we set $r_0(t) = \hat r$ for all $t\geq 0$, the process $(\psi_0(t))$ is governed by the same equation as $(\hat \psi(t))$. 
Therefore, it follows from ii) that $\projpf^{S^1} \rho_0' = \hat \rho$ and, thus, $\rho_0'=  \delta_{\hat r} \times \hat \rho= \rho_0$.
\end{proof}

We will now prove two lemmas, using properties of the stationary measures established in Proposition~\ref{prop:uniquestat},
that will be essential for our proof of Theorem~\ref{theo:main}.
The first lemma concerns a concentration bound for the measure $\xi_\epsilon(r)\dd r$ for sufficiently small $\epsilon$.
\begin{lemma}\label{lemm:concentration}
Let $f: \Rp \to \mathbb R$ be a function that satisfies $|f(r)| \lesssim 1+r^{-1}+r^k$ for some $k \in \N$. Then $f\in L^1_{\xi_\epsilon}$, for all $\epsilon >0$. Furthermore, if there exists an open neighborhood $U \subset \Rp$ of $\hat r$ with $f(r) = 0$ for all $r\in U$, then the bound
$$\left|\int_0^\infty f(r)\xi_\epsilon(r)\dd r\right|\leq \exp\left(-\frac{1}{\epsilon}\right)$$
holds for sufficiently small $\epsilon>0$. 
\end{lemma}
\begin{proof}
To show these statements we will make use of the substitution $p:= r^2$ and will write $\Tilde{f}(p) := f(r) = f(\sqrt{p})$.  
Firstly, we observe that
\begin{align*}
    \int_0^\infty |f(r)|\xi_\epsilon(r) \dd r =& \int_0^\infty |f(r)|L_\epsilon r \exp\left(-\frac{a}{2\epsilon^2\sigma'^2}\left(r^2-\Hat{
    r}^2\right)^2\right) \dd r\\
    =& \frac{L_\epsilon}{2}\int_0^\infty |\Tilde{f}(p)|  \exp\left(-\frac{a}{2\epsilon^2\sigma'^2}\left(p-\Hat{
    r}^2\right)^2\right) \dd p\\
    =& \frac{\sqrt{\pi}\epsilon \sigma'}{\sqrt{2a}}L_\epsilon  \mathbb E_{p \sim \mathcal N(\hat r^2, \epsilon^2\sigma'^2a^{-1})}\left[|\Tilde{f}(p)|\right]\\
    =& \frac{2}{\erfc\left(-\frac{\alpha}{\epsilon \sigma'\sqrt{2a}}\right)}  \mathbb E_{p \sim \mathcal N(\hat r^2, \epsilon^2\sigma'^2a^{-1})}\left[|\Tilde{f}(p)|\right],
\end{align*}
where $\mathcal N(\hat r^2, \epsilon^2\sigma'^2a^{-1})$ denotes a normal distribution with mean $\hat r^2$ and variance $\epsilon^2\sigma'^2a^{-1}$ and we interpret $\Tilde f$ as a function defined on the entire real line by setting $\Tilde f(p) := 0$ for $p \leq 0$. 
Since the complementary error function always takes values larger than 1 for negative arguments, we get the bound
$$ \int_0^\infty |f(r)|\xi_\epsilon(r) \dd r\leq 2 \mathbb E_{p \sim \mathcal N(\hat r^2, \epsilon^2\sigma'^2a^{-1})}\left[|\Tilde{f}(p)|\right].$$
Therefore, in order to show $f \in L^1_{\xi_\epsilon}$, it is sufficient to show $\Tilde f \in L^1_{\mathcal N(\hat r^2, \epsilon^2\sigma'^2a^{-1})}$. Suppose $f$ satisfies $f(r) \lesssim 1+r^{-1}+r^k$ and thus equivalently $\Tilde f(p) \lesssim 1+p^{-\frac{1}{2}}+p^{\frac{k}{2}}$. 
Then, we have $$\mathds 1_{(0,1]}(p) \Tilde f(p) \lesssim 1 + p^{-\frac{1}{2}},$$
which implies $\mathds 1_{(0,1]}\Tilde f \in L^1_{\operatorname{Leb}}$ and in particular $\mathds 1_{(0,1]}\Tilde f \in L^1_{\mathcal N(\hat r^2, \epsilon^2\sigma'^2a^{-1})}$. 
Similarly, we have
$$\mathds 1_{(1,\infty)}(p) \Tilde f(p) \lesssim 1 + p^{\frac{k}{2}},$$
which immediately implies $\mathds 1_{(1,\infty)}\Tilde f \in L^1_{\mathcal N(\hat r^2, \epsilon^2\sigma'^2a^{-1})}$. Since we have $\Tilde f = \mathds 1_{(0,1]}\Tilde f + \mathds 1_{(1,\infty)}\Tilde f$, we also have $\Tilde f \in L^1_{\mathcal N(\hat r^2, \epsilon^2\sigma'^2a^{-1})}$, which shows the first part of the lemma.

For the second part, suppose $f(r) = 0$ for all $r \in U$, where $U$ is an open neighborhood of $\hat r$. Then $\Tilde{f}$ will vanish in a neighborhood around $\hat r ^2$. Let $0 < \delta < \hat r^2$ be such that $\Tilde f(p) = 0$ for all $p \in (\hat r^2 -\delta, \hat r^2 +\delta)$. 
For $p \in \Rp \setminus (\hat r^2 -\delta, \hat r^2 +\delta)$, we have
\begin{align*}
    \frac{\frac{\sqrt{a}}{\epsilon\sigma'\sqrt{2\pi}}\exp\left(-\frac{a(p-\hat r ^2)^2}{2\epsilon^2\sigma'^2}\right)}{\frac{\sqrt{a}}{\sigma'\sqrt{2\pi}}\exp\left(-\frac{a(p-\hat r ^2)^2}{2\sigma'^2}\right)} &= \frac{1}{\epsilon}\exp\left(-\frac{a(p-\hat r ^2)^2}{2\sigma'^2}\left(\frac{1}{\epsilon^2}-1\right)\right)\\
    &\leq \frac{1}{\epsilon}\exp\left(-\frac{a\delta^2}{2\sigma'^2}\left(\frac{1}{\epsilon^2}-1\right)\right).
\end{align*}
Now, we can estimate
\begin{align*}
    \left|\int_0^\infty f(r)\xi_\epsilon(r)\dd r\right| &\leq \int_0^\infty |f(r)|\xi_\epsilon(r) \dd r\\
    &\leq 2 \frac{\sqrt{a}}{\epsilon\sigma'\sqrt{2\pi}}\int_0^\infty |\Tilde f(r)|\exp\left(-\frac{a(p-\hat r ^2)^2}{2\epsilon^2\sigma'^2}\right) \dd p\\
    &\leq 2\frac{1}{\epsilon}\exp\left(-\frac{a\delta^2}{2\sigma'^2}\left(\frac{1}{\epsilon^2}-1\right)\right)\frac{\sqrt{a}}{\sigma'\sqrt{2\pi}}\int_0^\infty |\Tilde f(r)|\exp\left(-\frac{a(p-\hat r ^2)^2}{2\sigma'^2}\right) \dd p\\
    &= 2\frac{1}{\epsilon}\exp\left(-\frac{a\delta^2}{2\sigma'^2}\left(\frac{1}{\epsilon^2}-1\right)\right) \mathbb E_{p \sim \mathcal N(\hat r^2, \sigma'^2a^{-1})}\left[|\Tilde{f}(p)|\right].
\end{align*}
By the first part of the lemma, we have $\Tilde f \in L^1_{\mathcal N(\hat r^2, \sigma'^2a^{-1})}$ such that the expectation in the last line is finite. 
Thus, for sufficiently small $\epsilon>0$, we get the desired bound
 \begin{align*}
     \left|\int_0^\infty f(r)\xi_\epsilon(r)\dd r\right|& \leq2\frac{1}{\epsilon}\exp\left(-\frac{a\delta^2}{2\sigma'^2}\left(\frac{1}{\epsilon^2}-1\right)\right) \mathbb E_{p \sim \mathcal N(\hat r^2, \sigma'^2a^{-1})}\left[|\Tilde{f}(p)|\right]\\
     &\leq \exp\left(-\frac{1}{\epsilon}\right).
 \end{align*}
 This finishes the proof.
\end{proof}

The second lemma concerns the weak convergence of $\rho_{\epsilon}$ to $\rho_0$ which will be crucial for obtaining continuity of the top Lyapunov exponents for $\epsilon \to 0$.
\begin{lemma}\label{lemm:statMeasConv}
The measures $\rho_\epsilon$ converge weakly in $\mathcal{M}_1(\Rp \times S^1)$ to $\rho_0 = \Hat{\rho} \times \delta_{\Hat{r}}$ as $\epsilon$ tends to zero.
\end{lemma}
The proof of this statement is inspired by \cite[Proposition 5.9]{VianaLyap}.
\begin{proof}[Proof of Lemma \ref{lemm:statMeasConv}]
By using formula \eqref{eq:xi}, we obtain directly that $(\projpf^{\Rpz} \rho_\epsilon) = (\xi_\epsilon(r)\dd r)$ converges weakly in $\mathcal M_1(\Rp)$ to $\delta_{\Hat{r}}$ as $\epsilon$ tends to zero, i.e.~weak convergence in the $r$-direction. Thus, it is sufficient to show that $(\projpf^{S^1} \rho_\epsilon)$ converges weakly in $\mathcal M_1 (S^1)$ to $\Hat{\rho}$. The space $\mathcal M_1 (S^1)$ is compact, by the Banach-Alaoglu theorem, and metrizable. Therefore, $(\projpf^{S^1} \rho_\epsilon)$ converges weakly to $\Hat{\rho}$ if and only if every accumulation point of $(\projpf^{S^1} \rho_\epsilon)$, as $\epsilon$ tends to zero, is equal to $\Hat{\rho}$.

Suppose now that $\nu \in \mathcal M_1(S^1)$ is such an accumulation point and $(\epsilon_n)_{n \in \N}$ is chosen such that $\epsilon_n \to 0$ and $\projpf^{S^1}\rho_{\epsilon_n} \to \nu$ weakly. The latter property is equivalent to $(\rho_{\epsilon_n})$ converging weakly to $\delta_{\Hat{r}} \times \nu$, which in turn also implies that $(\rho_{\epsilon_n} \times \delta_{\epsilon_n})$ converges weakly in $\mathcal M_1(\Rp \times S^1 \times \Rpz)$ to $\delta_{\Hat{r}} \times \nu \times \delta_{0}$. Let $\eta \in \mathcal C^0(\Rp\times S^1)$ be an arbitrary bounded continuous function. We have
\begin{align*}
    \int_{\Rp\times S^1} \eta(r,\psi) \nu(\dd \psi) \delta_{\Hat{r}}(\dd r) =& \lim_{n \to \infty} \int_{\Rp\times S^1} \eta(r,\psi) \rho_{\epsilon_n}(\dd r, \dd \psi)\\
    =& \lim_{n \to \infty} \int_{\Rp\times S^1} \mathcal P_{\epsilon_n}\eta(r,\psi) \rho_{\epsilon_n}(\dd r, \dd \psi)\\
    =&  \lim_{n \to \infty} \int_{\Rp \times S^1\times \Rpz} \mathcal P_{\epsilon}\eta(r,\psi) \rho_{\epsilon_n}(\dd r, \dd \psi)\delta_{\epsilon_n}(\dd \epsilon).
\end{align*}
Now we can use Lemma \ref{lemm:Pphi} i) and the weak convergence $\rho_{\epsilon_n} \times \delta_{\epsilon_n} \to \delta_{\Hat{r}} \times \nu \times \delta_{0}$ to get
\begin{align*}
    \int_{\Rp\times S^1\times \Rpz} \eta(r,\psi) \delta_{\Hat{r}}(\dd r)\nu(\dd \psi)  =&  \lim_{n \to \infty} \int_{\Rp\times S^1} \mathcal P_{\epsilon}\eta(r,\psi) \rho_{\epsilon_n}(\dd r,\dd \psi)\delta_{\epsilon_n}(\dd \epsilon)\\
    =&\int_{\Rp\times S^1\times \Rpz} \mathcal P_{\epsilon}\eta(r,\psi) \delta_{\Hat{r}}(\dd r)\nu(\dd \psi)\delta_{0}( \dd \epsilon)\\
    =& \int_{\Rp\times S^1} \mathcal P_0\eta(r,\psi) \delta_{\Hat{r}}(\dd r)\nu(\dd \psi)\\
    =& \int_{\Rp\times S^1} \mathcal \eta(r,\psi) \mathcal P^*_0 (\delta_{\Hat{r}} \times \nu)(\dd r, \dd \psi).
\end{align*}
Since $\eta \in C^0(\Rp\times S^1)$ was arbitrary, this implies $\delta_{\Hat{r} \times \nu} = \mathcal P^*_0 (\delta_{\Hat{r}} \times \nu)$. However, by Proposition \ref{prop:uniquestat}, the measure $\delta_{\Hat{r}}\times\Hat{\rho} $ is the unique fixed point of $\mathcal P_0^*$, so we must have $\nu = \Hat{\rho}$, completing the proof.
\end{proof}

\subsection{Furstenberg-Khasminskii formula for limiting process}
\label{sec:fk_formula}
Recall from Section~\ref{sec:sdlc} the Lyapunov exponent $\lambdahat$ (see \eqref{eq:topLyap_conv}) of \eqref{sde:ogsdlc}, the function $\Psi$ given by \eqref{def:mu} and the change of parameters
$$\alphahat = 2 \alpha,~\bhat = 2\frac{\alpha b'}{a} \text{ and } \sigmahat = \sigma' \sqrt{\frac{a}{\alpha}}.$$
For the limiting process on $S^1$, as given in Proposition~\ref{prop:fk}, we can prove the following Furstenberg-Khasminskii formula. A similar formula for a more general situation is already given in \cite{ImkellerLederer}. For completeness of our arguments, we provide a self-sustained derivation here.
\begin{proposition}\label{prop:fkfsdlc}
 We have
 \begin{equation}
 \label{eq:FKformula_limit}
     2\alpha~ \Psi\left(\frac{b'^2\sigma'^2}{2\alpha^2a}\right) = \lambdahat\left(2 \alpha, 2\frac{\alpha b'}{a}, \sigma'\sqrt{\frac{a}{\alpha}}\right) = \int_{S_1} h_3 (\hat r, \psi, 0) \hat \rho(\dd \psi).
 \end{equation}
\end{proposition}
In order to proof this proposition, we require the following lemma.
\begin{lemma}\label{lemm:noiselim}
Suppose $(w(t))$ is a real-valued semi-martingale satisfying $$\limsup_{t \to \infty} \frac{1}{t}\int_0^t w(t')^2\dd t' < \infty.$$
Then, for $(W(t))$ denoting some Brownian motion, we have 
$$\lim_{t \to \infty} \frac{1}{t}\int_0^t w(t') \dd W(t') = 0.$$
\end{lemma}
\begin{proof}
Let $(w(t))$ satisfy the assumption. Define a process $(X(t))$ by 
$$X(t) := \int_0^t w(t') \dd W(t').$$
Clearly $(X(t))$ is a local martingale with quadratic variation $\langle X(t) \rangle$ given by 
$$\langle X(t) \rangle = \int_0^t w(t')^2 \dd t'.$$
By the theorem of Dambis, Dubins-Schwarz (see e.g. \cite[Theorem V.1.7]{RevuzYor}) there exists a Brownian motion $(W_X(t))$ (which is not adapted to the original filtration and possibly even defined on an extension of the probability space), such that
$$W_X\left(\langle X(t)\rangle\right) = X(t).$$
By well-known growth bounds on the Brownian motion, this implies
\begin{align*}
    \lim_{t \to \infty} \frac{1}{t}X(t) &= \lim_{t \to \infty} \frac{1}{t} W_X(\langle X(t)\rangle)\\
    &\leq \limsup_{t \to \infty} \frac{\langle X(t)\rangle}{t} \lim_{t \to \infty} \frac{1}{\langle X(t)\rangle}W_X(\langle X(t)\rangle)\\
    &= \limsup_{t \to \infty} \frac{1}{t} \int_0^t w(t')^2\dd t' \lim_{t \to \infty} \frac{1}{t}W_X(t)\\
    &=0.
\end{align*}
\end{proof}
\begin{proof}[Proof of Proposition \ref{prop:fkfsdlc}]
From Theorem \ref{theo:LE-sdlc} we get
$$2\alpha~ \Psi\left(\frac{b'^2\sigma'^2}{2\alpha^2a}\right) = \lambdahat(\alphahat, \bhat, \sigmahat) = \lim_{t \to \infty} \frac{1}{t} \log\left(\sqrt{\hat s(t)^2 + \hat \theta(t)^2}\right) = \lim_{t \to \infty} \frac{1}{t} \hat \Lambda(t).$$
Using the SDE representation from Proposition \ref{prop:fk} yields
$$\lim_{t \to \infty} \frac{1}{t} \hat \Lambda(t) = \lim_{t \to \infty} \frac{1}{t}\left[\int_0^t  h_3(\hat r,\hat \psi(t'),  0) \dd t' + \int_0^t  h_2(\hat r,\hat \psi(t'),  0) \dd \hat W_3(t')\right].$$
Since the function $h_2(\hat r,\cdot,  0)$ is bounded, the process $(h_2(\hat r,\hat \psi(t'),  0))$ satisfies the assumption of Lemma \ref{lemm:noiselim} and we obtain
$$\lim_{t \to \infty} \frac{1}{t} \int_0^t h_2(\hat r,\hat \psi(t') , 0) \dd \hat W_3(t') = 0.$$
Ergodicity now gives us
$$\lambdahat(\alphahat, \bhat, \sigmahat) =  \lim_{t \to \infty} \frac{1}{t} \hat \Lambda(t) = \lim_{t \to \infty} \frac{1}{t}\int_0^t h_3(\hat r,\hat \psi(t'),  0) \dd t' = \int_{S_1} h_3 (\hat r,\psi,  0) \hat \rho(\dd \psi),$$
since $\hat\rho$ is the unique stationary distribution of the Markov process $(\hat \psi(t))$ by Proposition \ref{prop:uniquestat} (ii).
\end{proof}

\subsection{Continuity of Lyapunov exponents}
\label{sec:continuity_LEs}
Our goal is to show the limit 
\begin{equation}\label{eq:goallim}
    \lim_{\epsilon \to 0} \lambda(\alpha, \beta, a, \epsilon^{-1}b', \epsilon\sigma') = 2\alpha ~\Psi\left( \frac{b'^2\sigma'^2}{2\alpha^2a} \right).
\end{equation}
For that purpose, we want to find a Furstenberg-Khasminskii formula for the Lyapunov exponent on the left hand side, similar to the expression for the right-hand side given in equation~\eqref{eq:FKformula_limit}.
While such a formula has been already given in \cite[Equation (19)]{deville}, we encounter a subtle issue here. 
In order to show \eqref{eq:goallim} along the formal outline in Section \ref{sec:formal derivation}, we express the variational process in $(s_{\epsilon}, \theta_{\epsilon})$-coordinates (cf.~\eqref{sde:eps}), which are transformed into the projective coordinate $\psi_{\epsilon}$ (see Proposition~\ref{prop:fk}).
Analogously to the proof of Proposition \ref{prop:fkfsdlc}, we can write
\begin{align*}
    \lambda(\alpha, \beta, a, \epsilon^{-1}b', \epsilon\sigma') &= \lim_{t \to \infty} \frac{1}{t} \Lambda_\epsilon(t)\\
    &= \lim_{t \to \infty} \frac{1}{t}\left[\int_0^t h_3(r_\epsilon(t'), \psi_\epsilon(t'), \epsilon) \dd t' +\int_0^t h_2(r_\epsilon(t'), \psi_\epsilon(t'), \epsilon) \dd W_\phi(t')\right].
\end{align*}
Now we would like to use Lemma \ref{lemm:noiselim} in order to show that the second integral is negligible in the limit. 
By ergodicity, the assumption of Lemma \ref{lemm:noiselim} is equivalent to $h_2(\cdot, \cdot, \epsilon) \in L^2_{\rho_\epsilon}$, where $h_2$ is given by~\eqref{eq:defigihi}.
Since it seems out of reach to make explicit statements about the distribution $\rho_\epsilon$ beyond the fact that its marginal in the $r$-direction has density $\xi_\epsilon$, one is forced to rely on the estimate $h_2(r, \psi, \epsilon) \leq \sup_{\psi'}h_2(r, \psi', \epsilon)$, which leads to the bound  $|h_2(r,\psi, \epsilon)| \lesssim 1+\epsilon^{-2}r^{-1}$ given in Lemma \ref{lemm:bounds}. 
This, however, is not sufficient for obtaining the assumption of Lemma \ref{lemm:noiselim} since we clearly have $(r \mapsto 1+\epsilon^{-2}r^{-1}) \notin L^1_{\xi_\epsilon}$. Analogously, for the first integral above, we would hope to have $h_3(\cdot, \cdot, \epsilon) \in L^1_{\rho_\epsilon}$ in order to use ergodicity; however, we encounter a very similar problem.

We want to point out two aspects of this issue. Firstly, in both cases the problem arises due to the singularity at $r=0$. 
Secondly, neither of these problems has been present in previous works like \cite{deville}, as they only arise in the rescaled coordinates $(s_\epsilon, \theta_\epsilon) = (s_\epsilon, \epsilon\vartheta_\epsilon)$ (and their polar representation $(\psi_\epsilon, \Lambda_\epsilon)$) 
and are not present in the original coordinates $(s_\epsilon, \vartheta_\epsilon)$ (and their polar representation $(\Tilde\psi_\epsilon, \Tilde\Lambda_\epsilon)$, see Proposition~\ref{prop:fktilde}). 
However, the coordinate rescaling is crucial to obtain the aspired limit, as outlined in section \ref{sec:formal derivation}. 
In order to fix this issue we will introduce coordinates for the variational process (\ref{sde:vari}) which behave like $(s_\epsilon, \vartheta_\epsilon)$ whenever $r_\epsilon$ is close to 0 and behave like $(s_\epsilon, \theta_\epsilon)$ whenever $r_\epsilon$ is close to $\hat r$. Thereby we avoid the integrabitlity problem around the singularity at the origin, while having the ``correct'' coordinate system for the $\epsilon \to 0$-limit for a significant portion of time. We are now ready to finally prove Theorem~\ref{theo:main}.

\begin{proof}[Proof of Theorem \ref{theo:main}]
Note that Theorem \ref{theo:bif} already covers the second part of the statement. Thus it only remains to show \eqref{eq:mainLimit}.

We start by partitioning the positive real axis into four intervals $I_1 \cup I_2 \cup I_3 \cup I_4 = \Rp$ defined by 
\begin{align*}
    I_1&=\left(0,\frac{1}{3}\hat r\right], & I_2&=\left(\frac{1}{3}\hat r,\frac{2}{3}\hat r\right], & I_3&=\left(\frac{2}{3}\hat r, 2 \hat r\right],  & I_4&=\left(2\hat r, \infty\right).
\end{align*}
Next, we define a function $\chi: \Rp \to [0,1]$ (see Figure~\ref{fig:chiplot}) by
$$\chi(r) := \begin{cases}
0& \text{ if } r \in I_1,\\
\frac{3}{\hat r}r-1& \text{ if }r \in I_2,\\
1& \text{ if } r \in I_3 \cup I_4,
\end{cases}$$
and a process $(\Lambda_\epsilon^*(t))$ by
$$\Lambda^*_\epsilon(t) := [1-\chi(r_\epsilon(t))]\Tilde \Lambda_\epsilon(t) + \chi(r_\epsilon(t))\Lambda_\epsilon(t).$$
\begin{figure}
    \centering
    \begin{tikzpicture}
 
\begin{axis}[
    xmin = -0.5, xmax = 7.05,
    ymin = -0.6, ymax = 2,
    axis x line= center,
    axis y line= center,
    xtick distance = 10,
    ytick distance = 20,
    width = 0.9\textwidth,
    height = 0.4\textwidth,
    xlabel = {$r$},
    ylabel = {},
    axis line style={draw=none}]
    \draw[->,thin, black] (0,0)--(7.05,0);
\draw[->,thin, black] (0,0)--(0,1.5);
    \addplot[ thin,red] file[skip first] {xidata.dat};

\draw[-, thin, black, dashed] (1,0)--(1,1.5);
\draw[-, thin, black, dashed] (2,0)--(2,1.5);
\draw[-, thin, black] (3,0)--(3,1.5);
\draw[-, thin, black, dashed] (6,0)--(6,1.5);
\node[anchor=north](r1) at (axis cs: 1,0) {$\frac{1}{3}\hat r$};
 \node[anchor=north](r2) at (axis cs: 2,0) {$\frac{2}{3}\hat r$};
 \node[anchor=north](r3) at (axis cs: 3,0) {$\hat r$};
 \node[anchor=north](r4) at (axis cs: 6,0) {$2\hat r$};
 \draw[-, blue, line width = 0.5mm] (0,0)--(1,0);
 \draw[-, blue, line width = 0.5mm] (1,0)--(2,1);
 \draw[-, blue, line width = 0.5mm] (2,1)--(7,1);
    \draw [decorate,
    decoration = {brace}] (0.05,1.5) --  (0.95,1.5) node [midway, above] {$I_1$};
    \draw [decorate,
    decoration = {brace}] (1.05,1.5) --  (1.95,1.5) node [midway, above] {$I_2$};
    \draw [decorate,
    decoration = {brace}] (2.05,1.5) --  (5.95,1.5) node [midway, above] {$I_3$};
    \draw [decorate,
    decoration = {brace}] (6.05,1.5) --  (7.1,1.5) node [midway, above] {$I_4$};
 \node[anchor=north](text1) at (axis cs: 5,1) {\color{blue}$\chi(r)$}; 
 \node[anchor=west](text2) at (axis cs: 3.3,0.5) {\color{red}$\xi_\epsilon{}(r)$};
\end{axis}
 
\end{tikzpicture}
    \caption{The function $\chi$ plotted against the density $\xi_\epsilon$.}
    \label{fig:chiplot}
\end{figure}
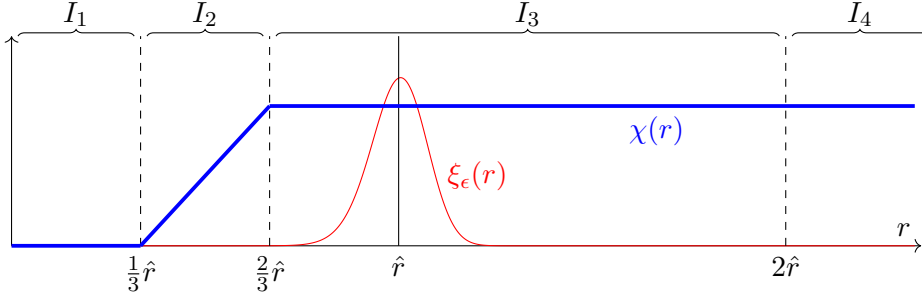
By (\ref{eq:FTLElim}) we have
$$\lambda(\alpha, \beta, a, \epsilon^{-1} b', \epsilon \sigma') = \lim_{t \to \infty}\frac{1}{t} \Lambda_\epsilon(t)=\lim_{t \to \infty}\frac{1}{t} \Tilde \Lambda_\epsilon(t),$$
and thus also
$$\lambda(\alpha, \beta, a, \epsilon^{-1} b', \epsilon \sigma') = \lim_{t \to \infty}\frac{1}{t} \Lambda_\epsilon^*(t).$$
Our goal now is to show
\begin{equation}\label{eq:proofgoal}
    \lim_{\epsilon \to 0} \frac{1}{t} \Lambda_\epsilon^*(t) = 2\alpha~ \Psi\left(\frac{b'^2\sigma'^2}{2\alpha^2a}\right),
\end{equation}
which will complete the proof.
By the definition of $(\Lambda^*_\epsilon(t))$, Propositions \ref{prop:fk}, \ref{prop:fktilde} and equation (\ref{sde:polar2}), we have
\begin{align*}
    \Lambda^*_\epsilon(t) =&\int_0^t [1-\chi(r_\epsilon(t'))] \dd \Tilde \Lambda_\epsilon(t') +\int_0^t \chi(r_\epsilon(t')) \dd \Lambda_\epsilon(t')\\
   &+\int_0^t \chi'(r_\epsilon(t'))\left(\Lambda_\epsilon(t')-\Tilde \Lambda_\epsilon(t')\right) \dd  r_\epsilon(t') \\
   =&\underbrace{\int_0^t [1-\chi(r_\epsilon(t'))]\Tilde h_1(r_\epsilon(t'),\Tilde \psi_\epsilon(t'),  \epsilon) \dd t'}_{=:\RN 2(t)}+
   \underbrace{\int_0^t \chi(r_\epsilon(t')) h_3( r_\epsilon(t'),\psi_\epsilon(t'),  \epsilon) \dd t'}_{=:\RN 1(t)}\\
   &+\underbrace{\int_0^t \chi(r_\epsilon(t')) h_2( r_\epsilon(t'),\psi_\epsilon(t'),  \epsilon)\dd W_\phi(t')}_{=:\RN 4(t)}\\
   &+ \underbrace{\int_0^t  \frac{3}{\hat{r}}\mathds 1_{I_2 }(r_\epsilon(t'))\left(\Lambda_\epsilon(t')-\Tilde \Lambda_\epsilon(t')\right) \left(\alpha r_\epsilon(t) - a r_\epsilon(t)^3 + \frac{\epsilon^2\sigma'^2}{2r_\epsilon(t)}\right) \dd  t'}_{=:\RN 3(t)}\\
   &+\underbrace{\int_0^t  \frac{3}{\hat{r}}\mathds 1_{I_2 }(r_\epsilon(t'))\left(\Lambda_\epsilon(t')-\Tilde \Lambda_\epsilon(t')\right) \epsilon\sigma' \dd W_r(t')}_{=:\RN 5(t)}.
\end{align*}
Here, we have made use of the fact that the quadratic covariations
$\langle r_\epsilon, \Lambda_\epsilon \rangle = \langle r_\epsilon, \Tilde \Lambda_\epsilon \rangle = 0$ vanish. 
In the following, we will determine the corresponding limits for the time averages of $\RN 1(t), \dots, \RN 5(t)$ separately. In particular we will show
$$\lim_{\epsilon \to 0}\lim_{t \to \infty} \frac{1}{t} \RN 1(t) = 2\alpha~ \Psi\left(\frac{b'^2\sigma'^2}{2\alpha^2a}\right),$$
as well as
$$\lim_{\epsilon \to 0}\limsup_{t \to \infty} \frac{1}{t} |\RN 2(t)| =\lim_{\epsilon \to 0}\limsup_{t \to \infty} \frac{1}{t} |\RN 3(t)| =\lim_{\epsilon \to 0}\limsup_{t \to \infty} \frac{1}{t} |\RN 4(t)| =\lim_{\epsilon \to 0}\limsup_{t \to \infty} \frac{1}{t} |\RN 5(t)| = 0.$$
Note that together these equations imply \eqref{eq:proofgoal}, thereby finishing the proof.

i) $\mathbf{\RN 1(t)}$: we  split the integral into
\begin{align*}
    \RN 1(t) =& \int_0^t \chi(r_\epsilon(t'))\mathds 1_{I_2\cup I_3}(r_\epsilon(t')) h_3( r_\epsilon(t'),\psi_\epsilon(t'),  \epsilon) \dd t' \\
    &+ \int_0^t \chi(r_\epsilon(t'))\mathds 1_{I_4}(r_\epsilon(t')) h_3( r_\epsilon(t'),\psi_\epsilon(t'),  \epsilon) \dd t'.
\end{align*}
First we will bound the second summand. Using Lemma \ref{lemm:bounds} ii) we get
\begin{align*}
   \left|\int_0^t \chi(r_\epsilon(t'))\mathds 1_{I_4}(r_\epsilon(t')) h_3(r_\epsilon(t'), \psi_\epsilon(t'),  \epsilon) \dd t'\right|
    &\leq \int_0^t \mathds 1_{I_4}(r_\epsilon(t')) \left|h_3(r_\epsilon(t'), \psi_\epsilon(t'),  \epsilon)\right| \dd t'\\
    &\lesssim \int_0^t \mathds 1_{I_4}(r_\epsilon(t')) \left(1+r_\epsilon(t')^2+\epsilon^{-4}r_\epsilon(t')^{-2}\right) \dd t'\\
    &\lesssim \int_0^t \mathds 1_{I_4}(r_\epsilon(t'))\epsilon^{-4} r_\epsilon(t')^2 \dd t'
\end{align*}
The function $f(r) := \mathds 1_{I_4}(r) r^2$ satisfies all assumptions of Lemma \ref{lemm:concentration}. Thus we can use ergodicity get the following estimate.
\begin{align*}
    &\phantom{=~}\lim_{\epsilon \to 0}\limsup_{t \to \infty}\frac{1}{t}\left|\int_0^t \chi(r_\epsilon(t'))\mathds 1_{I_4}(r_\epsilon(t')) h_3( r_\epsilon(t'),\psi_\epsilon(t'),  \epsilon) \dd t'\right|\\
    &\lesssim \lim_{\epsilon \to 0}\epsilon^{-4} \lim_{t \to \infty} \frac{1}{t} \int_0^t \mathds 1_{I_4}(r_\epsilon(t'))r_\epsilon(t')^2\dd t'\\
    &=\lim_{\epsilon \to 0} \epsilon^{-4}\int_0^\infty \mathds 1_{I_4}(r) r^2 \xi_\epsilon(r) \dd r\\
    &\leq \lim_{\epsilon \to 0} \epsilon^{-4}\exp\left(-\frac{1}{\epsilon}\right)\\
    &=0.
\end{align*}

Concerning the first summand, we obtain by ergodicity
\begin{align*}
    \lim_{\epsilon \to 0} \lim_{t \to \infty} \frac{1}{t} \RN 1(t) &= \lim_{\epsilon \to 0} \lim_{t \to \infty} \frac{1}{t}\int_0^t \chi(r_\epsilon(t'))\mathds 1_{I_2\cup I_3}(r_\epsilon(t')) h_3(r_\epsilon(t'), \psi_\epsilon(t'),  \epsilon) \dd t'\\
    &= \lim_{\epsilon \to 0} \int_{(I_2\cup I_3)\times S^1} \chi(r) h_3( r,\psi,  \epsilon) \rho_\epsilon(\dd r,\dd \psi ).
\end{align*}
Since $(I_2\cup I_3)\times S^1$ is a compact domain, the function $h_3(\cdot, \cdot, \epsilon)$ converges uniformly to $h_3(\cdot, \cdot, 0)$ as $\epsilon \to 0$. Using this fact together with the weak convergence established in Lemma \ref{lemm:statMeasConv} and the expression \eqref{eq:FKformula_limit}, we get
\begin{align*}
    \lim_{\epsilon \to 0} \lim_{t \to \infty} \frac{1}{t} \RN 1(t) &= \lim_{\epsilon \to 0} \int_{(I_2\cup I_3)\times S^1} \chi(r) h_3( r,\psi,  \epsilon) \rho_\epsilon(\dd r,\dd \psi)\\
    &= \int_{(I_2\cup I_3)\times S^1} \chi(r) h_3( r,\psi,  0) \delta_{\hat r}(\dd r)\hat \rho(\dd \psi) \\
    &=\int_{S^1} h_3( \hat r,\psi,  0) \hat \rho(\dd \psi)\\
    &=2\alpha ~\Psi\left( \frac{b'^2\sigma'^2}{2\alpha^2a} \right).
\end{align*}
 ii) $\mathbf{\RN 2(t)}$: similarly to the argument for the second summand in i), we  first use Lemma \ref{lemm:bounds} iii) and then ergodicity to get
\begin{align*}
    \lim_{\epsilon \to 0}\limsup_{t \to \infty}\frac{1}{t}|\RN 2 (t)| &= \lim_{\epsilon \to 0}\limsup_{t \to \infty}\frac{1}{t}\left|\int_0^t [1-\chi(r_\epsilon(t'))]\Tilde h_1(r_\epsilon(t'),\Tilde \psi_\epsilon(t'),  \epsilon) \dd t'\right|\\
    &\leq \lim_{\epsilon \to 0}\limsup_{t \to \infty}\frac{1}{t}\int_0^t [1-\chi(r_\epsilon(t'))]\left|\Tilde h_1(r_\epsilon(t'),\Tilde \psi_\epsilon(t'),  \epsilon) \right|\dd t'\\
    &\lesssim \lim_{\epsilon \to 0}\lim_{t \to \infty}\frac{1}{t}\int_0^t [1-\chi(r_\epsilon(t'))] \frac{1}{\epsilon}r_\epsilon(t)^2\dd t'\\
    &=\lim_{\epsilon \to 0}\frac{1}{\epsilon} \int_0^\infty [1-\chi(r)]r^2\xi_\epsilon(r) \dd r.
\end{align*}
Again, the function $f(r) := [1-\chi(r)]r^2$ satisfies all the assumptions in Lemma \ref{lemm:concentration} and we obtain
$$\lim_{\epsilon \to 0}\limsup_{t \to \infty}\frac{1}{t}|\RN 2 (t)|\leq \lim_{\epsilon \to 0}\frac{1}{\epsilon} \int_0^\infty [1-\chi(r)]r^2\xi_\epsilon(r) \dd r \leq\lim_{\epsilon \to 0}\frac{1}{\epsilon}\exp\left(-\frac{1}{\epsilon}\right)= 0.$$

iii) $\mathbf{\RN 3 (t)}$: note that the term 
$$\mathds 1_{I_2 }(r_\epsilon(t))\left(\alpha r_\epsilon(t) - a r_\epsilon(t)^3 + \frac{\epsilon^2\sigma'^2}{2r_\epsilon(t)}\right)$$
is bounded uniformly in $\epsilon$ for $\epsilon \in (0,1]$. Together with the bound \eqref{ineq:tilde}, this allows for the estimate
\begin{align*}
    &\phantom{=~}\lim_{\epsilon\to 0} \limsup_{t \to \infty} \frac{1}{t}|\RN 3 (t)| \\
    &=\lim_{\epsilon\to 0} \limsup_{t \to \infty} \frac{1}{t} \left|\int_0^t  \frac{3}{\hat{r}}\mathds 1_{I_2 }(r_\epsilon(t'))\left(\Lambda_\epsilon(t')-\Tilde \Lambda_\epsilon(t')\right) \left(\alpha r_\epsilon(t) - a r_\epsilon(t)^3 + \frac{\epsilon^2\sigma'^2}{2r_\epsilon(t)}\right) \dd  t'\right|\\
    &\lesssim \lim_{\epsilon\to 0} \limsup_{t \to \infty} \frac{1}{t} \int_0^t \mathds 1_{I_2 }(r_\epsilon(t'))\left|\Lambda_\epsilon(t')-\Tilde \Lambda_\epsilon(t')\right|\dd  t'\\
    &\leq \lim_{\epsilon\to 0} -\log(\epsilon) \int_0^\infty \mathds 1_{I_2 }(r) \xi_\epsilon(r)\dd  r.
\end{align*}
Now we can again use Lemma \ref{lemm:concentration} to obtain
\begin{align*}
    \lim_{\epsilon\to 0} \limsup_{t \to \infty} \frac{1}{t}|\RN 3 (t)|&\leq \lim_{\epsilon\to 0} -\log(\epsilon) \int_0^\infty \mathds 1_{I_2 }(r) \xi_\epsilon(r)\dd  r\\
    &\leq \lim_{\epsilon\to 0} -\log(\epsilon) \exp\left(-\frac{1}{\epsilon}\right)\\
    &= 0.
\end{align*}

iv) $\mathbf{\RN 4(t)}$: as a consequence of Lemma \ref{lemm:bounds} (i), we have
$$\chi(r)h_2(r,\psi,  \epsilon) \leq \mathds 1_{I_2\cup I_3 \cup I_4}(r)h_2(r,\psi,  \epsilon) \lesssim \mathds 1_{I_2\cup I_3 \cup I_4}(r)(1 + \epsilon^{-2}r^{-1}) \lesssim 1 + \epsilon^{-2}.$$
Thus, in particular, the process $(\chi(r_\epsilon(t)) h_2(  r_\epsilon(t), \psi_\epsilon(t),\epsilon))$ is bounded for every $\epsilon>0$. Therefore Lemma \ref{lemm:noiselim} gives
$$\lim_{t \to \infty} \frac{1}{t} \RN 4 (t) = \lim_{t \to \infty} \frac{1}{t}\int_0^t \chi(r_\epsilon(t')) h_2( r_\epsilon(t'), \psi_\epsilon(t'), \epsilon)\dd W_\phi(t') = 0,$$
for each $\epsilon>0.$

v) $\mathbf{\RN 5(t)}$:  using the estimate~\eqref{ineq:tilde}, the process $(\mathds 1_{I_2 }(r_\epsilon(t'))(\Lambda_\epsilon(t')-\Tilde \Lambda_\epsilon(t')))$ is bounded. Hence, similarly to iv), Lemma \ref{lemm:noiselim} yields
$$\lim_{t \to \infty} \frac{1}{t} \RN 5 (t) = \frac{3\epsilon\sigma'}{\hat{r}}\lim_{t \to \infty} \frac{1}{t}\int_0^t  \mathds 1_{I_2 }(r_\epsilon(t'))\left(\Lambda_\epsilon(t')-\Tilde \Lambda_\epsilon(t')\right)\dd W_r(t') = 0,$$
for each $\epsilon>0.$
\end{proof}

\appendix
\section{Polar coordinates}\label{app:polar}
This appendix contains the proofs for Propositions \ref{prop:polarSDE} and \ref{prop:variSDE} and for Lemma \ref{lemm:proj}.
\begin{proof}[Proof of Proposition \ref{prop:polarSDE}]
First we will derive Statonovish SDEs for $(r(t))$ and $(\phi(t))$. Note that, by definition,
$$r(t) = \|Z(t)\| = \sqrt{Z_1(t)^2+Z_2(t)^2}.$$
By the chain-rule for Stratonovich SDEs, we have
\begin{align*}
    \dd r(t) =& \frac{Z_1(t)}{r(t)}\circ \dd Z_1(t) + \frac{Z_2(t)}{r(t)}\circ \dd Z_2(t) \\
    =& \frac{Z_1(t)}{r(t)}\left[\alpha Z_1(t) - \beta Z_2(t) - ar(t)^2Z_1(t) +br(t)^2Z_2(t)\right]\dd t +\sigma\frac{Z_1(t)}{r(t)}\circ \dd W_1(t)\\
    &+\frac{Z_2(t)}{r(t)}\left[\beta Z_1(t) + \alpha Z_2(t) - br(t)^2Z_1(t) -ar(t)^2Z_2(t)\right]\dd t +\sigma\frac{Z_2(t)}{r(t)}\circ \dd W_2(t)\\
    =& \frac{1}{r(t)}\left[\alpha \left(Z_1(t)^2 + Z_2(t)^2\right) - ar(t)^2\left(Z_1^2 + Z_2^2\right)\right]\dd t \\
    &+ \sigma [\cos(\phi(t))\circ \dd W_1(t) + \sin(\phi(t))\circ \dd W_2(t)]\\
    =& \left(\alpha r(t) -ar(t)^3\right)\dd t + \sigma [\cos(\phi(t))\circ \dd W_1(t) + \sin(\phi(t))\circ \dd W_2(t)].
\end{align*}
For $(\phi(t))$ we first show
\begin{align}\label{eq:radderi}
    \frac{-Z_2(t)}{r(t)^2}\circ \dd Z_1(t) + \frac{Z_1(t)}{r(t)^2} \circ \dd Z_2(t) =& \frac{-\sin(\phi(t))}{r(t)}\left[-r(t)\sin(\phi(t))\circ \dd \phi(t) +\cos(\phi(t))\dd r(t)\right]\nonumber\\
    &+\frac{\cos(\phi(t))}{r(t)}\left[r(t)\cos(\phi(t))\circ \dd \phi(t) +\sin(\phi(t))\dd r(t)\right]\nonumber\\
    =& \sin^2(\phi(t))\circ \dd \phi(t) + \cos^2(\phi(t))\circ \dd \phi(t) \nonumber\\
    =& \dd \phi(t).
\end{align}
Now, using the chain-rule again, we can compute
\begin{align*}
    \dd \phi(t) =& \frac{-Z_2(t)}{r(t)^2}\circ \dd Z_1(t) + \frac{Z_1(t)}{r(t)^2} \circ \dd Z_2(t) \\
    =& \frac{-Z_2(t)}{r(t)^2}\left[\alpha Z_1(t) - \beta Z_2(t) - ar(t)^2Z_1(t) +br(t)^2Z_2(t)\right]\dd t - \sigma \frac{Z_2(t)}{r(t)}\circ \dd W_1(t)\\
    &+ \frac{Z_1(t)}{r(t)^2}\left[\beta Z_1(t) + \alpha Z_2(t) - br(t)^2Z_1(t) -ar(t)^2Z_2(t)\right]\dd t + \sigma \frac{Z_1(t)}{r(t)}\circ \dd W_2(t)\\
    =& \frac{1}{r(t)^2}\left[\beta\left(Z_1(t)^2+Z_2(t)^2\right) - br(t)^2\left(Z_1(t)^2+Z_2(t)^2\right)\right] \dd t \\
    &+ \sigma\left[-\sin(\phi(t)) \circ \dd W_1(t) + \cos(\phi(t))\circ \dd W_2(t)\right]\\
    =&\left(\beta-br(t)^2\right)\dd t + \sigma\left[-\sin(\phi(t)) \circ \dd W_1(t) + \cos(\phi(t))\circ \dd W_2(t)\right].
\end{align*}
It remains to compute the relevant quadratic co-variations for the It\^o-Statonovish correction terms. Using the notation $\langle \cdot, \cdot \rangle$ for the quadratic co-variation of two semi-martingales, we get
\begin{align*}
    \dd \left\langle\cos(\phi(t)), W_1(t)\right\rangle =& -\sin(\phi(t))\dd\langle\phi(t), W_1(t)\rangle \\
    =& \frac{\sigma}{r(t)}\sin^2(\phi(t))\dd t,\\
    \dd \left\langle\sin(\phi(t)), W_2(t)\right\rangle =& \cos(\phi(t))\dd\langle\phi(t), W_2(t)\rangle \\
    =& \frac{\sigma}{r(t)}\cos^2(\phi(t))\dd t,\\
    \dd \left\langle\frac{\sin(\phi(t))}{r(t)}, W_1(t)\right\rangle =& \frac{\cos(\phi(t))}{r(t)}\dd\left\langle\phi(t), W_1(t)\right\rangle - \frac{\sin(\phi(t))}{r(t)^2}\dd\left\langle r(t), W_1(t)\right\rangle\\
    =& -\frac{\sigma}{r(t)^2}\cos(\phi(t))\sin(\phi(t))\dd t - \frac{\sigma}{r(t)^2}\cos(\phi(t))\sin(\phi(t))\dd t\\
    =&-\frac{2\sigma}{r(t)^2}\cos(\phi(t))\sin(\phi(t))\dd t,\\
    \dd \left\langle\frac{\cos(\phi(t))}{r(t)}, W_2(t)\right\rangle =& -\frac{\sin(\phi(t))}{r(t)}\dd\left\langle\phi(t), W_2(t)\right\rangle -\frac{\cos(\phi(t))}{r(t)^2}\dd \langle r(t),W_2(t)\rangle \\
    =& -\frac{\sigma}{r(t)^2}\cos(\phi(t))\sin(\phi(t))\dd t - \frac{\sigma}{r(t)^2}\cos(\phi(t))\sin(\phi(t))\dd t,\\
    =&-\frac{2\sigma}{r(t)^2}\cos(\phi(t))\sin(\phi(t))\dd t.
\end{align*}
Finally, we can use these to obtain It\^o-SDEs for $(r(t))$ and $(\phi(t))$. We have
\begin{align*}
    \dd r(t) =& \left(\alpha r(t) -ar(t)^3\right)\dd t + \sigma [\cos(\phi(t))\circ \dd W_1(t) + \sin(\phi(t))\circ \dd W_2(t)] \\
    =&\left(\alpha r(t) -ar(t)^3\right)\dd t +\frac{\sigma}{2r(t)} \left[\dd \langle \cos(\phi(t)), W_1(t)\rangle + \dd \langle \sin(\phi(t)), W_2(t)\rangle \right] \\
    &+ \sigma [\cos(\phi(t)) \dd W_1(t) + \sin(\phi(t)) \dd W_2(t)] \\
    =&\left(\alpha r(t) -ar(t)^3\right)\dd t +\frac{\sigma}{2} \left[\frac{\sigma}{r(t)} \sin^2(\phi(t))\dd t+\frac{\sigma}{r(t)} \cos^2(\phi(t))\dd t\right] \\
    &+ \sigma [\cos(\phi(t)) \dd W_1(t) + \sin(\phi(t)) \dd W_2(t)] \\
    =&\left(\alpha r(t) -ar(t)^3 + \frac{\sigma^2}{2r(t)}\right)\dd t + \sigma [\cos(\phi(t)) \dd W_1(t) + \sin(\phi(t)) \dd W_2(t)],
\end{align*}
as well as,
\begin{align*}
    \dd \phi(t) =& \left(\beta-br(t)^2\right)\dd t + \sigma\left[-\sin(\phi(t)) \circ \dd W_1(t) + \cos(\phi(t))\circ \dd W_2(t)\right] \\
     =&  \left(\beta-br(t)^2\right)\dd t + \frac{\sigma}{2}\left[-\dd\left\langle \frac{\sin(\phi(t))}{r(t)}, W_1(t)\right\rangle + \dd \left\langle\frac{\cos(\phi(t))}{r(t)},W_2(t)\right\rangle\right]\\
    &+ \sigma\left[-\sin(\phi(t)) \dd W_1(t) + \cos(\phi(t))\dd W_2(t)\right] \\
    =&  \left(\beta-br(t)^2\right)\dd t + \frac{\sigma}{2}\left[-\frac{2\sigma}{r(t)^2}\cos(\phi(t))\sin(\phi(t))\dd t+\frac{2\sigma}{r(t)^2}\cos(\phi(t))\sin(\phi(t))\dd t\right]\\
    &+ \sigma\left[-\sin(\phi(t)) \dd W_1(t) + \cos(\phi(t))\dd W_2(t)\right] \\
    =&  \left(\beta-br(t)^2\right)\dd t
    + \sigma\left[-\sin(\phi(t)) \dd W_1(t) + \cos(\phi(t))\dd W_2(t)\right].
\end{align*}
This finishes the proof.
\end{proof}

\begin{proof}[Proof of Proposition \ref{prop:variSDE}]
Recall that $(s(t))$ and $(\vartheta(t))$ were defined by
$$s(t) := \begin{pmatrix}\cos(\phi(t)) \\ \sin(\phi(t))\end{pmatrix}^T Y(t)$$
and
$$\vartheta(t) := \begin{pmatrix}-\sin(\phi(t)) \\ \cos(\phi(t))\end{pmatrix}^T Y(t).$$
By the integration by parts formula for Stratonovich-integrals, we get
$$\dd s(t) = \begin{pmatrix}\cos(\phi(t)) \\ \sin(\phi(t))\end{pmatrix}^T \circ\dd Y(t) + Y(t)^T \circ \dd\begin{pmatrix}\cos(\phi(t)) \\ \sin(\phi(t))\end{pmatrix}.$$
For the sake of readability, we will compute the summands separately. Firstly, we have
\begin{align*}
    \begin{pmatrix}\cos(\phi(t)) \\ \sin(\phi(t))\end{pmatrix}^T \circ\dd Y(t) 
    =&\begin{pmatrix}\cos(\phi(t)) \\ \sin(\phi(t))\end{pmatrix}^T\begin{pmatrix}\alpha & -\beta \\ \beta & \alpha \end{pmatrix}Y(t) \dd t\\
    &- \|Z(t)\|^2\begin{pmatrix}\cos(\phi(t)) \\ \sin(\phi(t))\end{pmatrix}^T\begin{pmatrix}a & -b \\ b & a \end{pmatrix}Y(t) \dd t\\ 
    &-2\begin{pmatrix}\cos(\phi(t)) \\ \sin(\phi(t))\end{pmatrix}^T \begin{pmatrix}a & -b \\ b & a \end{pmatrix}Z(t)Z(t)^TY(t) \dd t \\
    =&\alpha \begin{pmatrix}\cos(\phi(t)) \\ \sin(\phi(t))\end{pmatrix}^TY(t)\dd t - \beta \begin{pmatrix}-\sin(\phi(t)) \\ \cos(\phi(t))\end{pmatrix}^T Y(t)\dd t\\
    &-r(t)^2a \begin{pmatrix}\cos(\phi(t)) \\ \sin(\phi(t))\end{pmatrix}^TY(t)\dd t + r(t)^2b \begin{pmatrix}-\sin(\phi(t)) \\ \cos(\phi(t))\end{pmatrix}^T Y(t)\dd t\\
    &-2r(t)^2\begin{pmatrix}\cos(\phi(t)) \\ \sin(\phi(t))\end{pmatrix}^T \begin{pmatrix}a & -b \\ b & a \end{pmatrix}\begin{pmatrix}\cos(\phi(t)) \\ \sin(\phi(t))\end{pmatrix}\begin{pmatrix}\cos(\phi(t)) \\ \sin(\phi(t))\end{pmatrix}^TY(t) \dd t\\
    =& \left(\alpha s(t) - \beta \vartheta(t)-ar(t)^2s(t)+br(t)^2\vartheta(t)-2ar(t)^2s(t)\right)\dd t\\
    =& \left(\alpha s(t) - \beta \vartheta(t)-3ar(t)^2s(t)+br(t)^2\vartheta(t)\right)\dd t.
\end{align*}
Secondly, we obtain
\begin{align*}
     Y(t)^T \circ \dd\begin{pmatrix}\cos(\phi(t)) \\ \sin(\phi(t))\end{pmatrix} &= \begin{pmatrix}-\sin(\phi(t)) \\ \cos(\phi(t))\end{pmatrix}^T Y(t)\circ \dd \phi(t)\\
     &= \left(\beta -br(t)^2\right)\vartheta(t)\dd t +\frac{\sigma}{r(t)}\vartheta(t)\circ\dd W_\phi(t).
\end{align*}
Together this gives
$$\dd s(t) = \left(\alpha-3ar(t)^2\right)s(t)\dd t+\frac{\sigma}{r(t)}\vartheta(t)\circ\dd W_\phi(t).$$
We proceed analogously for $(\vartheta(t))$. Integration by parts gives
$$\dd \vartheta(t) = \begin{pmatrix}-\sin(\phi(t)) \\ \cos(\phi(t))\end{pmatrix}^T \circ\dd Y(t) + Y(t)^T \circ \dd\begin{pmatrix}-\sin(\phi(t)) \\ \cos(\phi(t))\end{pmatrix}.$$
Computing the summands separately again we get
\begin{align*}
    \begin{pmatrix}-\sin(\phi(t)) \\ \cos(\phi(t))\end{pmatrix}^T \circ\dd Y(t) &= \begin{pmatrix}-\sin(\phi(t)) \\ \cos(\phi(t))\end{pmatrix}^T\begin{pmatrix}\alpha & -\beta \\ \beta & \alpha \end{pmatrix}Y(t) \dd t\\
    &- \|Z(t)\|^2\begin{pmatrix}-\sin(\phi(t)) \\ \cos(\phi(t))\end{pmatrix}^T\begin{pmatrix}a & -b \\ b & a \end{pmatrix}Y(t) \dd t\\ 
    &-2\begin{pmatrix}-\sin(\phi(t)) \\ \cos(\phi(t))\end{pmatrix}^T \begin{pmatrix}a & -b \\ b & a \end{pmatrix}Z(t)Z(t)^TY(t) \dd t \\
    =&\alpha \begin{pmatrix}-\sin(\phi(t)) \\ \cos(\phi(t))\end{pmatrix}^TY(t)\dd t + \beta \begin{pmatrix}\cos(\phi(t)) \\ \sin(\phi(t))\end{pmatrix}^T Y(t)\dd t\\
    &-r(t)^2a \begin{pmatrix}-\sin(\phi(t)) \\ \cos(\phi(t))\end{pmatrix}^TY(t)\dd t - r(t)^2b \begin{pmatrix}\cos(\phi(t)) \\ \sin(\phi(t))\end{pmatrix}^T Y(t)\dd t\\
    &-2r(t)^2\begin{pmatrix}-\sin(\phi(t)) \\ \cos(\phi(t))\end{pmatrix}^T \begin{pmatrix}a & -b \\ b & a \end{pmatrix}\begin{pmatrix}\cos(\phi(t)) \\ \sin(\phi(t))\end{pmatrix}\begin{pmatrix}\cos(\phi(t)) \\ \sin(\phi(t))\end{pmatrix}^TY(t) \dd t\\
    =& \left(\alpha \vartheta(t) + \beta s(t)-ar(t)^2\vartheta(t)-br(t)^2s(t)-2br(t)^2s(t)\right)\dd t\\
    =& \left(\alpha \vartheta(t) + \beta s(t)-ar(t)^2\vartheta(t)-3br(t)^2s(t)\right)\dd t,
\end{align*}
as well as,
\begin{align*}
     Y(t)^T \circ \dd\begin{pmatrix}-\sin(\phi(t)) \\ \cos(\phi(t))\end{pmatrix} &= -\begin{pmatrix}\cos(\phi(t)) \\ \sin(\phi(t))\end{pmatrix}^T Y(t)\circ \dd \phi(t)\\
     &= -\left(\beta -br(t)^2\right)s(t)\dd t -\frac{\sigma}{r(t)}s(t)\circ\dd W_\phi(t).
\end{align*}
Together this gives
$$\dd\vartheta(t) = \left(\alpha- ar(t)^2\right) \vartheta(t) \dd t - 2br(t)^2 s(t)\dd t- \frac{\sigma}{r(t)} s(t) \circ \dd W_\phi(t),$$
which finishes the proof.
\end{proof}
\begin{proof}[Proof of Lemma \ref{lemm:proj}]
Recall that $(\psi(t))$ and $(\Lambda(t))$ are defined by 
\begin{align}\label{eq:psilambdadefi}
    \psi(t) &:= 2 \tan^{-1}\left(\frac{v_2(t)}{v_1(t)}\right), & \Lambda(t) &:= \log\left(\sqrt{v_1(t)^2 + v_2(t)^2}\right).
\end{align}
For ease of notation, we will simply write $\psi,\Lambda, v_1, \dots$ instead of $\psi(t), \Lambda(t), v_1(t), \dots$ in the following.
Note that we have
\begin{align*}
    \cos\left(\frac{1}{2}\psi\right) =& \cos\left(\tan^{-1}\left(\frac{v_2}{v_1}\right)\right)\\
    =&\left(\frac{v_2^2}{v_1^2}+1\right)^{-\frac{1}{2}}\\
    =&\frac{|v_1|}{\sqrt{v_1^2+v_2^2}}
\end{align*}
and
\begin{align*}
    \sin\left(\frac{1}{2}\psi\right) =& \sin\left(\tan^{-1}\left(\frac{v_2}{v_1}\right)\right)\\
    =&\frac{v_2}{v_1}\left(\frac{v_2^2}{v_1^2}+1\right)^{-\frac{1}{2}}\\
    =&\frac{|v_1|v_2}{v_1\sqrt{v_1^2+v_2^2}}.
\end{align*}
This allows us to derive the identities
\begin{align*}
    1+\cos(\psi) =& 2 \cos^2\left(\frac{1}{2}\psi\right)\\
    =& \frac{v_1^2}{\sqrt{v_1^2+v_2^2}},
\end{align*}
\begin{align*}
    1-\cos(\psi) =& 2 \sin^2\left(\frac{1}{2}\psi\right)\\
    =& \frac{v_2^2}{\sqrt{v_1^2+v_2^2}}
\end{align*}
and
\begin{align*}
    \sin(\psi) =& 2 \cos\left(\frac{1}{2}\psi\right)\sin\left(\frac{1}{2}\psi\right)\\
    =& \frac{v_1v_2}{\sqrt{v_1^2+v_2^2}}.
\end{align*}
Applying the chain rule to (\ref{eq:psilambdadefi}) yields
\begin{align*}
    \dd \psi =& - \frac{2v_2}{v_1^2+v_2^2} \circ \dd v_1 +  \frac{2v_1}{v_1^2+v_2^2} \circ \dd v_2\\
    =& - \frac{2v_2}{v_1^2+v_2^2} \left(B^{(1)}_{1,1}(w)v_1 + B^{(1)}_{1,2}(w)v_2\right)\dd t - \frac{2v_2}{v_1^2+v_2^2} \left(B^{(2)}_{1,1}(w)v_1 + B^{(2)}_{1,2}(w)v_2\right)\circ\dd W\\
    &+ \frac{2v_1}{v_1^2+v_2^2} \left(B^{(1)}_{2,1}(w)v_1 + B^{(1)}_{2,2}(w)v_2\right)\dd t + \frac{2v_1}{v_1^2+v_2^2} \left(B^{(2)}_{2,1}(w)v_1 + B^{(2)}_{2,2}(w)v_2\right)\circ\dd W\\
    =&\frac{2}{{v_1^2+v_2^2}}\left(-B^{(1)}_{1,1}(w)v_1v_2  -B^{(1)}_{1,2}(w)v_2^2+B^{(1)}_{2,1}(w)v_1^2 + B^{(1)}_{2,2}(w)v_1 v_2\right)\dd t\\
    &+\frac{2}{{v_1^2+v_2^2}}\left(-B^{(2)}_{1,1}(w)v_1v_2 -B^{(2)}_{1,2}(w)v_2^2+B^{(2)}_{2,1}(w)v_1^2 + B^{(2)}_{2,2}(w)v_1 v_2\right)\circ\dd W\\
    =&\left[B^{(1)}_{2,1}(w)(1+\cos(\psi) - B^{(1)}_{1,2}(w)(1-\cos(\psi) + \left(B^{(1)}_{2,2}(w)-B^{(1)}_{1,1}(w)\right)\sin(\psi(t))\right]\dd t\\
    &+\left[B^{(2)}_{2,1}(w)(1+\cos(\psi) - B^{(2)}_{1,2}(w)(1-\cos(\psi) + \left(B^{(2)}_{2,2}(w)-B^{(2)}_{1,1}(w)\right)\sin(\psi(t))\right]\circ\dd W
\end{align*}
and
\begin{align*}
    \dd \Lambda =& \frac{v_1}{v_1^2+v_2^2} \circ \dd v_1 +  \frac{v_2}{v_1^2+v_2^2} \circ \dd v_2\\
    =& \frac{v_1}{v_1^2+v_2^2} \left(B^{(1)}_{1,1}(w)v_1 + B^{(1)}_{1,2}(w)v_2\right)\dd t + \frac{v_1}{v_1^2+v_2^2} \left(B^{(2)}_{1,1}(w)v_1 + B^{(2)}_{1,2}(w)v_2\right)\circ\dd W\\
    &+ \frac{v_2}{v_1^2+v_2^2} \left(B^{(1)}_{2,1}(w)v_1 + B^{(1)}_{2,2}(w)v_2\right)\dd t + \frac{v_2}{v_1^2+v_2^2} \left(B^{(2)}_{2,1}(w)v_1 + B^{(2)}_{2,2}(w)v_2\right)\circ\dd W\\
    =&\frac{1}{{v_1^2+v_2^2}}\left(B^{(1)}_{1,1}(w)v_1^2 + B^{(1)}_{1,2}(w)v_1v_2+B^{(1)}_{2,1}(w)v_1v_2 + B^{(1)}_{2,2}(w)v_2^2\right)\dd t\\
    &+\frac{1}{{v_1^2+v_2^2}}\left(B^{(2)}_{1,1}(w)v_1^2 + B^{(2)}_{1,2}(w)v_1v_2+B^{(2)}_{2,1}(w)v_1v_2 + B^{(2)}_{2,2}(w)v_2^2\right)\circ\dd W\\
    =&\left[B^{(1)}_{1,1}(w)(1+\cos(\psi) + B^{(1)}_{2,2}(w)(1-\cos(\psi) + \left(B^{(1)}_{1,2}(w)+B^{(1)}_{2,1}(w)\right)\sin(\psi(t))\right]\dd t\\
    &+\left[B^{(2)}_{1,1}(w)(1+\cos(\psi) + B^{(2)}_{2,2}(w)(1-\cos(\psi) + \left(B^{(2)}_{1,2}(w)+B^{(2)}_{2,1}(w)\right)\sin(\psi(t))\right]\circ\dd W.
\end{align*}
This finishes the proof.
\end{proof}
\section*{Acknowledgements}
The authors thank the DFG SPP 2298 for supporting their research. Both authors have been additionally supported by Germany’s Excellence Strategy – The Berlin Mathematics Research Center MATH+ (EXC-2046/1, project ID:
390685689), in the case of D.C. via the Berlin Mathematical School. Furthermore, M.E. thanks the DFG CRC 1114 for support.
\bibliographystyle{abbrv}
\bibliography{biblio}

\end{document}